\documentclass[10pt]{article}   
\usepackage{graphicx}
\usepackage{latexsym,amsfonts,amsmath,theorem,amssymb}
\usepackage{tikz}
\usepackage{stmaryrd,array,tabularx,bbm}
\usepackage{amsmath,booktabs,ctable,threeparttable,boxedminipage}
\usepackage{pstricks,graphicx,epstopdf}
\usepackage{a4wide}
\usepackage{color}
\usepackage{caption}
\usepackage{theorem}
\usepackage[colorlinks=true,citecolor=red,urlcolor=blue,linkcolor=blue]{hyperref}
\usepackage{cleveref}
\usepackage{framed}
\usepackage{blindtext}
\usepackage{float}
\usepackage{todonotes}
\usepackage{lineno}
\usepackage{multirow}
\usepackage[caption=false]{subfig}
\usepackage{algorithm,algpseudocode}
\usepackage{qbordermatrix}
\usepackage{blkarray}
\usepackage{xcolor}
\usepackage{geometry}
\usepackage{blkarray}
\usepackage{arydshln}
\usepackage{lipsum}
\usepackage{amsfonts}
\usepackage{tikz-cd}  
\usepackage[caption=false]{subfig} 
\usepackage{arydshln}
\usepackage{amsopn}

\DeclareMathOperator{\diag}{diag}

\ifpdf
  \DeclareGraphicsExtensions{.eps,.pdf,.png,.jpg}
\else
  \DeclareGraphicsExtensions{.eps}
\fi

\algrenewcommand\algorithmicrequire{\textbf{Input:}}
\algrenewcommand\algorithmicensure{\textbf{Output:}}

\newcommand\argmin{\mathop{\mathrm{argmin}}}
\newcommand\calR{\mathcal{R}}
\newcommand\calN{\mathcal{N}}

\newcommand\calX{\mathcal{X}}
\newcommand\calY{\mathcal{Y}}
\newcommand\calA{\mathcal{A}}
\newcommand\calL{\mathcal{L}}
\newcommand\calP{\mathcal{P}}

\newcommand\calG{\mathcal{G}}
\newcommand\rA{[\mathcal{A}]}
\newcommand\rtA{[\mathcal{A}^{*}]}

\numberwithin{equation}{section}
\numberwithin{figure}{section}
\numberwithin{table}{section}

\newtheorem{theorem}{Theorem}[section]
\newtheorem{lemma}{Lemma}[section]
\newtheorem{proposition}{Proposition}[section]

\newtheorem{remark1}{Remark}[section]

\newenvironment{proof}{\begin{trivlist}
    \item[\hskip\labelsep{\bf Proof.}]}{$\hfill\Box$\end{trivlist}}
\newenvironment{proofof}[1]{\begin{trivlist}
    \item[\hskip\labelsep{\bf Proof of {#1}.}]}{$\hfill\Box$\end{trivlist}}
{\theoremstyle{plain} \theorembodyfont{\rmfamily}
}

\newcommand{\vertiii}[1]{{\left\vert\kern-0.25ex\left\vert\kern-0.25ex\left\vert #1 
\right\vert\kern-0.25ex\right\vert\kern-0.25ex\right\vert}}
\title{Characterizing GSVD by singular value expansion of linear operators and its computation}
\author{Haibo Li
\thanks{School of Mathematics and Statistics, The University of Melbourne, Parkville, VIC 3010, Australia.
\href{mailto:haibo.li@unimelb.edu.au}{haibo.li@unimelb.edu.au}} 
}
\date{}

\begin{document}
\maketitle

\begin{abstract}
	The generalized singular value decomposition (GSVD) of a matrix pair $\{A, L\}$ with $A\in\mathbb{R}^{m\times n}$ and $L\in\mathbb{R}^{p\times n}$ generalizes the singular value decomposition (SVD) of a single matrix. In this paper, we provide a new understanding of GSVD from the viewpoint of SVD, based on which we propose a new iterative method for computing nontrivial GSVD components of a large-scale matrix pair. By introducing two linear operators $\calA$ and $\calL$ induced by $\{A, L\}$ between two finite-dimensional Hilbert spaces and applying the theory of singular value expansion (SVE) for linear compact operators, we show that the GSVD of $\{A, L\}$ is nothing but the SVEs of $\calA$ and $\calL$. This result characterizes completely the structure of GSVD for any matrix pair with the same number of columns. As a direct application of this result, we generalize the standard Golub-Kahan bidiagonalization (GKB) that is a basic routine for large-scale SVD computation such that the resulting generalized GKB (gGKB) process can be used to approximate nontrivial extreme GSVD components of $\{A, L\}$, which is named the gGKB\_GSVD algorithm. We use the GSVD of $\{A, L\}$ to study several basic properties of gGKB and also provide preliminary results about convergence and accuracy of gGKB\_GSVD for GSVD computation. Numerical experiments are presented to demonstrate the effectiveness of this method.
\end{abstract}




\section{Introduction}
The generalized singular value decomposition (GSVD) of a matrix pair is an extension of the singular value decomposition (SVD) of a single matrix. First introduced by Van Loan \cite{Van1976} and further developed by many others \cite{Paige1981,Sun1983,Van1985}, now the GSVD has been a standard matrix decomposition \cite{Bjorck1996,Golub2013}. The GSVD provides an important mathematical tool for analyzing relationships between two sets of variables or matrices, which is particularly useful in various applications, including signal processing \cite{nakamura2012real,speiser1984signal}, statistics \cite{park2005relationship,paige1985general}, computational biology \cite{alter2003generalized}, and many others \cite{bhuyan2004application,ewerbring1989canonical,kaagstrom1984generalized,kuo2000applications,howland2003structure}.

Let $I_{k}$ denote the identity matrix of order $k$ and $\mathbf{0}$ denote the zero matrix or vector with dimensions clarified by the context. For any two matrices with the same number of columns, the general-form GSVD is stated as follows \cite{Paige1981}:
\begin{theorem}[GSVD]\label{thm_gsvd}
	Let $A\in\mathbb{R}^{m\times n}$ and $L\in\mathbb{R}^{p\times n}$ with $\mathrm{rank}((A^{\top},L^{\top})^{\top})=r$. Then the GSVD of $\{A, L\}$ is
	\begin{subequations}\label{GSVD1}
	\begin{equation}\label{gsvd1}
	  A = P_{A}C_AX^{-1} , \ \ \  L = P_{L}S_LX^{-1} ,
  \end{equation}
	  with
	  \begin{equation}\label{gsvd2}
		  C_A =
		  \bordermatrix*[()]{%
			  \Sigma_{A} & \mathbf{0} & m \cr
			  r &   n-r   \cr
		  } \ , \  \ \ 
		  S_L = \bordermatrix*[()]{%
			  \Sigma_{L} & \mathbf{0} & p \cr
			  r &   n-r   \cr
		  } 
	  \end{equation}
	  and
	  \begin{equation}\label{gsvd3}
		  \Sigma_{A} =
		  \bordermatrix*[()]{%
			  I_{q_1}  &  &  & q_1 \cr
			  &  C_{q_2}  &  & q_2 \cr
			  &  & \mathbf{0}  & m-q_1-q_2 \cr
			  q_1 & q_2 & q_3
		  } , \ \ \
		  \Sigma_{L} =
		  \bordermatrix*[()]{%
			  \mathbf{0}  &  &  & p-r+q_1 \cr
			  &  S_{q_2}  &  & q_2 \cr
			  &  & I_{q_3}  & q_3 \cr
			  q_1 & q_2 & q_3
		  },
	  \end{equation}	
	  \end{subequations}
	  where $q_1+q_2+q_3=r$, and  $P_{A}\in \mathbb{R}^{m\times m}$, $P_{L}\in \mathbb{R}^{p\times p}$ are orthogonal, $X\in\mathbb{R}^{n\times n}$ is invertible, and $\Sigma_{A}^{T}\Sigma_A+\Sigma_{L}^{T}\Sigma_L=I_{r}$. The values of $q_1$, $q_2$ and $q_3$ are defined internally by the matrices $A$ and $L$.
\end{theorem}

If $r=n$, then $\{A, L\}$ is called a regular matrix pair. Discussions about GSVD for regular and nonregular matrix pairs can be found in \cite{LiRC1993,Sun1983} and \cite{Paige1981,Sun1983b,Paige1984}, respectively. Write $C_{q_2}=\mbox{diag}(c_{q_1+1}, \dots, c_{q_1+q_2})$ with $1>c_{q_1+1}\geq \cdots \geq c_{q_1+q_2}>0$ and $S_{q_2}=\mbox{diag}(s_{q_1+1}, \dots, s_{q_1+q_2})$ with $0<s_{q_1+1}\leq \cdots \leq s_{q_1+q_2}<1$. Let $c_{1}=\cdots=c_{q_1}=1, \ c_{q_1+q_2+1}=\cdots=c_{r}=0$ and $s_{1}=\cdots=s_{q_1}=0, \ s_{q_1+q_2+1}=\cdots=s_{r}=1$. Write $X=(x_1,\ldots,x_n)$, $P_A=(p_{A,1},\ldots,p_{A,m})$ and $P_L=(p_{L,1},\ldots,p_{L,p})$. We call the tuple $(c_i, s_i, x_{i}, p_{A,i},p_{L,i})$ the $i$-th nontrivial GSVD components, and the $i$-th largest generalized singular value is $\gamma_{i}:=c_{i}/s_{i}$ satisfying $c_{i}^{2}+s_{i}^{2}=1$, where $1\leq i\leq r$. In this paper, we consider the nontrivial GSVD components and their computations.

Despite its remarkable capabilities, computing the GSVD poses significant challenges. Early computational approaches for the GSVD were built upon adaptations of algorithms designed for the SVD; for small-scale matrices, there are several such numerical algorithms for full GSVD computation \cite{Van1985,paige1986computing,bai1993computing}. Recent development on stable computation of the CS decomposition (CSD) \cite{sutton2012stable} provides another alternative for small-scale GSVD computation. For large and sparse problems, obtaining the full GSVD may not be feasible, yet it is often necessary to compute only a subset of GSVD components relevant to practical applications. Typically, this refers to certain extreme GSVD components, which are those with the largest or smallest corresponding generalized singular values, or interior GSVD components, which are those whose corresponding desired generalized singular values are closest to a specified target.

For large-scale GSVD computation, the first step is usually tranforming it as an equivalent generalized eigendecomposition (GED) problem \cite{Bai2000} or CSD problem. The joint bidiagonalization (JBD) method proposed by Zha \cite{Zha1996} can be used to compute a few extreme GSVD components, which is essentially an indirect procedure for CSD of the Q-factor of a regular $\{A, L\}$ (i.e. the Q matrix in the QR factorization). This method relies on a JBD process that iteratively reduces $\{A, L\}$ to bidiagonal forms simultaneously. Jia and Li \cite{jiali2021,jia2023joint} made a detailed analysis for the numerical behavior of the JBD method and the convergence behavior for extreme GSVD components in finite precision arithmetic. They proposed the semi-orthogonalization strategy and design a partial reorthogonalization procedure to maintain regular convergence of the computed approximate GSVD components. Subsequently, Li \cite{li2024joint} analyzed the influence of computational errors resulting from inaccurate inner iterations in JBD on the convergence and final accuracy of computed GSVD components and proposed a modified version of the JBD method. Recently, Alvarruiz et al. \cite{alvarruiz2024thick} developed a thick restart technique for JBD to compute a partial GSVD, enabling the storage and computation cost further saved. 
On the other hand, the Jacobi--Davidson type algorithms \cite{hochstenbach2006jacobi} are capable of computing a few interior GSVD components. A representative work is the Jacobi--Davidson GSVD (JDGSVD) method proposed by Hochstenbach \cite{hochstenbach2009jacobi}, which formulates the GSVD of $\{A, L\}$ as a GED problem of an augmented symmetric matrix pair. This method is further analyzed and developed in several subsequent work; see e.g. \cite{refahi2017new,huang2021choices,huang2022two}. As the development of contour integration technique for eigenvalue problems of finding interior eigenvalues \cite{sakurai2003projection,peter2014feast}, recently a contour integral-based algorithm has been adopted for interior GSVD components computation \cite{liu2023contour}.

Apart from regarding the GSVD as an equivalent CSD or GED, there is very little work on understanding and analyzing GSVD from other perspectives. In \cite{chu1997variational}, the authors studied the GSVD using a variational formulation analogous to that of the SVD, providing a new understanding of the generalized singular vectors. Recently, by treating $(A^{\top}, L^{\top})^{\top}$ (more precisely, the range space of it) as a point in the real Grassmann manifold $\mathrm{Gr}(m+p,r)$ --- the manifold composed of all $r$-dimensional subspaces of $\mathbb{R}^{m+p}$ --- the authors in \cite{edelman2020gsvd} interpreted a modified form of GSVD as a coordinate representation of $(A^{\top}, L^{\top})^{\top}$. For developing practical GSVD algorithms, however, these new perspectives on GSVD are not enough. It would be beneficial to understand the GSVD from the viewpoint of SVD so that many existing algorithms for large-scale SVD are possible to be adapted for large-scale GSVD computation. One well-known result is that $\{\gamma_i\}$ are the singular values of $AL^{\dag}$ if $L$ has full column rank \cite{Zha1996}, where ``$\dag$'' is the Moor-Penrose pseudoinverse. But for noninvertible and nonsquare $L$, generally the GSVD of $\{A, L\}$ is not related to the SVD of $AL^{\dag}$; this issue becomes much more complicated for nonregular matrix pairs.

In this paper, we provide a new understanding of GSVD from the viewpoint of SVD. This new perspective relies on the theory of singular value expansion (SVE) for linear compact operators \cite[\S 2.2]{Engl2000}, which is essentially the SVD if the compact operator is a matrix between two Euclidean spaces. Denote by $\calR(\cdot)$ and $\calN(\cdot)$ the range space and null space of a matrix, respectively. By defining the positive semidefinite matrix $M=A^{\top}A+L^{\top}L$, we first investigate the structure of trivial and nontrivial GSVD components $x_i$, showing that those trivial $\{x_i\}$ form a basis for $\calN(M)$ and any nontrivial $x_i$ belongs to the coset $\bar{x}_i+\calN(M)$, where $\bar{x}_i\in\calR(M)$ is a nontrivial GSVD component. Then we consider the nontrivial $x_i \in \calR(M)$ and other corresponding GSVD components. By introducing a linear operator $\calA$ induced by $\{A, L\}$ between two finite-dimensional Hilbert spaces, where a non-Euclidean inner-product is used, we show that the SVE of $\calA$ is just the nontrivial GSVD components of $A$, i.e. the first decomposition in \eqref{gsvd1}. Similarly, we introduce a linear operator $\calL$ induced by $\{A, L\}$ and show that the SVE of $\calL$ is just the nontrivial GSVD components of $L$. This result reveals that the nontrivial part of the GVSD of $\{A, L\}$ is nothing but the SVEs of the two linear operators induced by $\{A, L\}$. Combined with the trivial GSVD components $\{x_i\}$, it completely characterizes the structure of GSVD for any matrix pair with the same number of columns.

As a direct application of the above result, we propose a new iterative method for computing several extreme nontrivial GSVD components of $\{A, L\}$. This iterative method is a natural extension of the Golub-Kahan bidiagonalization (GKB) method for large-scale SVD computation \cite{Golub1965}, which iteratively reduces a matrix to a bidiagonal form by a Lanczos-type iterative process. There are several variants and extensions for the standard GKB of a single matrice, which are proposed to solve large-scale generalized least squares problems \cite{benbow1999solving,arioli2013generalized}, saddle point problems \cite{dumitrasc2023generalized}, or regularization of inverse problems \cite{chung2017generalized,chung2018efficient,li2023preconditioned,li2023subspace,li2024scalable}; most of them are named the \textit{generalized Golub-Kahan bidiagonalization} (gGKB) while some have other different names. As a natural analogy of the standard GKB for SVD computation, we develop an operator-type GKB for linear operators $\calA$ and $\calL$ to approximate their SVE components, which is also named the \textsf{gGKB} process. We derive matrix-form recursive relations for this operator-type GKB so that it can be used in practical computations. Moreover, this approach offers a unified and general treatment for extending the standard GKB, which can be used to derive nearly all of the aforementioned gGKB processes. 

Using the GSVD of $\{A, L\}$, we study several basic properties of the proposed \textsf{gGKB} process. Due to the correspondence of GSVD and SVE, the \textsf{gGKB} method can approximate well the extreme nontrivial GSVD components of $\{A, L\}$, resulting in the \textsf{gGKB\_GSVD} algorithm. We derive a relative residual norm and its sharp upper bound for the computed nontrivial GSVD components, which is a good measure of the approximating accuracy and can be used in a stopping criterion for practical computations. Several preliminary results about the convergence and accuracy of \textsf{gGKB\_GSVD} in exact arithmetic are provided, showing the effectiveness of this method.

The paper is organized as follows. In \Cref{sec2}, we review several basic properties of the GSVD. In \Cref{sec3}, we introduce two linear operators induced by $\{A, L\}$ to characterize the structure of GSVD by the SVEs of them. In \Cref{sec4} we propose the new \textsf{gGKB} process and study its basic properties. In \Cref{sec5}, we propose the \textsf{gGKB\_GSVD} algorithm for computing nontrivial GSVD components. Numerical experiments are presented in \Cref{sec6}, and concluding remarks follow in \Cref{sec7}.

\section{GSVD, SVD and Golub-Kahan bidiagonalization}\label{sec2}
We review several basic properties of the GSVD of $\{A, L\}$ presented in \Cref{thm_gsvd}. The nontrivial generalized singular values of $\{A, L\}$ in descending order are
\begin{equation}\label{gsvd_value1}
	\underbrace{\infty, \dots, \infty}_{q_1}, \ \
	\underbrace{c_{q_1+1}/s_{q_1+1}, \dots, c_{q_1+q_2}/s_{q_1+q_2}}_{q_2}, \ \ 
	\underbrace{0, \dots, 0}_{q_3}.
\end{equation}
We remark that the three numbers $q_1$, $q_2$, $q_3$ are uniquely determined by the property of $\{A, L\}$, and some of them may be zero in certain instances. The nontrivial GSVD components are linked by the vector-form relations
\begin{equation}\label{GSV}
	\begin{cases}
		Ax_i=c_i p_{A,i}  \\
		Lx_i=s_i p_{L,i}    \\
		s_i A^Tp_{A,i}=c_iL^{T}p_{L,i}  
	\end{cases}
\end{equation}
for $i=1,\dots,r$. For those trivial GSVD components, it holds that
\begin{equation}
	Ax_i = \mathbf{0}, \ \ Lx_i = \mathbf{0}, \ \ 
	A^Tp_{A,i} = \mathbf{0}, \ \ L^Tp_{L,i} = \mathbf{0} 
\end{equation}
for $i=r+1,\dots, n$. 
The following result describes the structure of trivial and nontrivial GSVD components $\{x_i\}$.

\begin{proposition}\label{lem:gsvd_new}
	Let $M=A^{\top}A+L^{\top}L$ and partition $X$ as $\bordermatrix*[()]{
		X_1 & X_2 & \cr
		r &  n-r   \cr
	}$. Then $\calR(X_2)=\calN(M)$. Moreover, let 
	\begin{equation}
		\bar{X}=(\bar{X}_1 \ \ X_2), \ \ \ \bar{X}_1=(\calP_{\calR(M)}x_1,\dots,\calP_{\calR(M)}x_r) .
	\end{equation}
	Then it holds
	\begin{equation}\label{gsvd_new}
		A = P_{A}C_{A}\bar{X}^{-1}, \ \ \ L = P_{L}S_{L}\bar{X}^{-1} .
	\end{equation}
\end{proposition}
\begin{proof}
	It is clear that $\calN(M)=\calN(A)\cap\calN(L)$, and using th GSVD of $\{A,L\}$, it is easy to obtain that $\calN(A)\cap\calN(L)=\calR(X_2)$. Thus, we have $\calR(X_2)=\calN(M)$. Using the relation that $\calP_{\calR(M)}x_i=x_i-\calP_{\calN(M)}x_i$, for $1\leq i\leq r$ we have
	\begin{equation*}
		A\calP_{\calR(M)}x_i = Ax_i-A\calP_{\calN(M)}x_i=Ax_i, \ \ 
		L\calP_{\calR(M)}x_i = Lx_i-L\calP_{\calN(M)}x_i=Lx_i .
	\end{equation*}
	Using the above two relations, it is easy to verify \eqref{gsvd_new}.
\end{proof}

Since $\mathrm{dim}(\calN(M))=n-r=\mathrm{rank}(\{x_{i}\}_{i=r+1}^{n})$, it follows that $\{x_{i}\}_{i=r+1}^{n}$ forms a basis for $\calN(M)$. \Cref{lem:gsvd_new} also indicates that for any $x_i$ with $1\leq i\leq r$, $\calP_{\calR(M)}x_i$ is also an $i$-th generalized singular vector of $\{A,L\}$. Therefore, any nontrivial $x_i$ can be decomposed into two components, one being $\calP_{\calR(M)}x_i\in\calR(M)$ and the other being an arbitrary vector in $\calN(M)$, which means that any nontrivial $x_i$ belongs to the coset $\bar{x}_i+\calN(M)$, where $\bar{x}_i\in\calR(M)$ is the $i$-th nontrivial GSVD component. In particular, we can focus on those nontrivial $x_i$ in $\calR(M)$, which results in the new form of GSVD \eqref{gsvd_new}. In the subsequent part, we always consider this form of GSVD by requiring
\begin{equation}\label{new_xi}
	x_i\in\calR(M) \ \ \mathrm{for} \ \ 1\leq i\leq r .
\end{equation}

There exists a direct relationship between the SVD and GSVD for a matrix pair with a special property. If $L$ has full column rank, it follows from \eqref{GSVD1} that $r=n$ and $q_1=0$, leading to
\begin{equation}\label{SVD1}
	AL^{\dag} = P_{A}C_{A}X^{-1}[(P_{L}S_{L})X^{-1}]^{\dag}
	= P_{A}C_{A}X^{-1}X(P_{L}S_{L})^{\dag} = P_{A}(C_{A}S_{L}^{\dag})P_{L}^{\top} ,
\end{equation}
where we have used the property that $(M_1M_2)^{\dag}=M_{2}^{\dag}M_{1}^{\dag}$ if $M_1$ has full column rank and $M_2$ has full row rank. Therefore, the SVD of $AL^{\dag}$ is $P_{A}(C_{A}S_{L}^{\dag})P_{L}^{\top}$ with the singular values be $\{\gamma_i\}_{i=1}^{n}$. 

For the above case, one can compute the SVD of $AL^{\dag}$ to get the GSVD components \footnote{For numerical computations, it is not recommended to explicitly form $AL^{\dag}$ due to its numerical unstability, especially when $L$ is close to column rank-deficient.}. For a large-scale matrix, the GKB process is a basic routine for computing a few extreme SVD components. At each iteration, it reduces the matrix to a lower-order bidiagonal matrix and generates two Krylov subspaces. The Rayleigh-Ritz procedure is then exploited to approximate extreme SVD components using the bidiagonal matrix and Krylov subspaces \cite{Bai2000}.

In the following part of the paper, we characterize the GSVD from the viewpoint of SVD for any matrix pair with the same number of columns. Then we generalize the GKB process so that it can be used to compute extreme GSVD components.

\section{Characterizing GSVD by singular value expansion of linear operators}\label{sec3}
We first discuss linear operators between two finite-dimensional Hilbert spaces. Then we study the GSVD of $\{A, L\}$ using the singular value expansion of linear operators. Note that \Cref{sec3.1} is quite general without requiring $M=A^{\top}A+L^{\top}L$.

\subsection{Linear operator induced by matrices}\label{sec3.1}
Suppose $G\in\mathbb{R}^{m\times m}$ is symmetric positive definite. It is obviously that $\langle u,u'\rangle_{G}:=u^{\top}Gu'$ defines an inner product on $\mathbb{R}^{m}$ such that $(\mathbb{R}^{m},\langle\cdot,\cdot\rangle_{G})$ is an $m$-dimensional Hilbert space. On the other hand, for any symmetric positive semidefinite matrix $M\in\mathbb{R}^{n\times n}$ with $\mathrm{rank}(M)=r$, the bilinear form $\langle v,v'\rangle_{M}:=v^{\top}Mv'$ is not a well-defined inner product on $\mathbb{R}^{n}$ if $r<n$. In this case, we consider the inner product on the subspace $\calR(M)$.

\begin{lemma}
	The bilinear form $\langle v,v'\rangle_{M}:=v^{\top}Mv'$ for any $v, v'\in\calR(M)$ is an inner product, and $(\calR(M),\langle\cdot,\cdot\rangle_{M})$ is an $r$-dimensional Hilbert space.
\end{lemma}
\begin{proof}
	The statement $\mathrm{dim}(\calR(M))=\mathrm{rank}(M)$ is a basic property. We need to show that $\langle \cdot,\cdot\rangle_{M}$ is indeed an inner product, i.e., it is a symmetric and positive bilinear form on $\calR(M)\times \calR(M)$. We only need to prove the positiveness. To see it, let $v\in\calR(M)$ such that $\langle v,v\rangle_{M}=v^{\top}Mv=0$. It follows $v\in\calN(M)$. Note that $\calR(M)\cap \calN(M)=\{\mathbf{0}\}$ since $M$ is symmetric, which leads to $v=\mathbf{0}$.
\end{proof}

Given a matrix $A\in\mathbb{R}^{m\times n}$. Define the linear map 
\begin{equation}\label{Amap}
	\calA: (\calR(M),\langle\cdot,\cdot\rangle_{M}) \rightarrow (\mathbb{R}^{m},\langle\cdot,\cdot\rangle_{G}), \ \ \ 
	v \mapsto A v ,
\end{equation}
where $v$ and $Av$ are column vectors under the canonical bases of $\mathbb{R}^{n}$ and $\mathbb{R}^{m}$. Let $W_r\in\mathbb{R}^{n\times r}$ and $Z\in\mathbb{R}^{m\times m}$ be two matrices whose columns constitute orthonormal bases of $(\calR(M),\langle\cdot,\cdot\rangle_{M})$ and $(\mathbb{R}^{m},\langle\cdot,\cdot\rangle_{G})$, respectively, i.e. $W_{r}^{\top}MW_{r}=I_{r}$ and $Z^{\top}GZ=I_{m}$. Then we have the commutative diagram:
\begin{equation}\label{commu_dig1}
	\begin{tikzcd}
		(\calR(M),\langle\cdot,\cdot\rangle_{M}) \arrow[r, "\calA"] \arrow[from=d, "\pi_1"]
		& (\mathbb{R}^{m},\langle\cdot,\cdot\rangle_{G}) \arrow[from=d, "\pi_2"'] \\
		(\mathbb{R}^{r},\langle\cdot,\cdot\rangle_{2}) \arrow[r, "\rA"']
		& (\mathbb{R}^{m},\langle\cdot,\cdot\rangle_{2}) \ ,
	\end{tikzcd}
\end{equation}
where $\rA$ denotes the matrix representation of $\calA$ under bases $W_r$ and $Z$, and $\pi_1$ and $\pi_2$ are two linear maps:
\begin{align}
	& \pi_{1}: (\mathbb{R}^{r},\langle\cdot,\cdot\rangle_{2}) \rightarrow (\calR(M),\langle\cdot,\cdot\rangle_{M}), \ \ \ 
	x \mapsto W_{r}x, \\
	& \pi_{2}: (\mathbb{R}^{m},\langle\cdot,\cdot\rangle_{2}) \rightarrow (\mathbb{R}^{m},\langle\cdot,\cdot\rangle_{G}), \ \ \  \ \ \ 
	y \mapsto Zy .
\end{align}
Note that $\pi_1$ and $\pi_2$ are two isomorphism maps such that  $(\calR(M),\langle\cdot,\cdot\rangle_{M})\cong (\mathbb{R}^{r},\langle\cdot,\cdot\rangle_{2})$ and $(\mathbb{R}^{m},\langle\cdot,\cdot\rangle_{G})\cong (\mathbb{R}^{m},\langle\cdot,\cdot\rangle_{2})$. Since $\calA$ is a bounded linear operator, we can define its adjoint 
\begin{equation}
	\calA^{*}: (\mathbb{R}^{m},\langle\cdot,\cdot\rangle_{G}) \rightarrow (\calR(M),\langle\cdot,\cdot\rangle_{M}), \ \ \ 
	u \mapsto \calA^{*}u
\end{equation}
by the relation 
\begin{equation}
	\langle \calA v, u \rangle_{G} = \langle \calA^{*} u, v \rangle_{M}
\end{equation}
for any $v\in(\calR(M),\langle\cdot,\cdot\rangle_{M})$ and $u\in(\mathbb{R}^{m},\langle\cdot,\cdot\rangle_{G})$. We have the following corresponding commutative diagram:
\begin{equation}\label{commu_dig2}
	\begin{tikzcd}
		(\calR(M),\langle\cdot,\cdot\rangle_{M}) \ar[from=r, "\calA^{*}"'] \ar[from=d, "\pi_1"]
		& (\mathbb{R}^{m},\langle\cdot,\cdot\rangle_{G}) \ar[from=d, "\pi_2"'] \\
		(\mathbb{R}^{r},\langle\cdot,\cdot\rangle_{2}) \ar[from=r, "\rtA"]
		& (\mathbb{R}^{m},\langle\cdot,\cdot\rangle_{2}) \ ,
	\end{tikzcd}
\end{equation}
where $\rtA$ is the matrix representation of $\calA^{*}$ under bases $W_r$ and $Z$. The following result gives the matrix-forms of $\rA$ and $\rtA$.

\begin{lemma}\label{lem:matrix_op}
	The matrix representations of $\calA$ and $\calA^{*}$ under bases $W_r$ and $Z$ are
	\begin{equation}
		\rA = Z^{-1}AW_{r}, \ \ \ 
		\rtA = W_{r}^{\top}A^{\top}GZ .
	\end{equation}
\end{lemma}
\begin{proof}
	By the commutative diagram \eqref{commu_dig1}, we have 
	$
		\calA\circ \pi_{1}(x) = \pi_{2}\circ\rA(x)
	$
	for any $x\in\mathbb{R}^{r}$, which is equivalent to
	$
		AW_{r}x = Z\rA x
	$
	Thus, we have $AW_{r} = Z\rA$, leading to $\rA = Z^{-1}AW_{r}$. Similarly, by the commutative diagram \eqref{commu_dig2}, we have 
	\begin{equation*}
		\calA^{*}\circ \pi_{2}(y) = \pi_{1}\circ\rtA(y) \ \ \Leftrightarrow \ \ 
		\calA^{*}Zy = W_{r}\rtA y
	\end{equation*}
	for any $y\in\mathbb{R}^{m}$, which leads to
	\begin{equation}\label{mat_At}
		\calA^{*}Z = W_{r}\rtA .
	\end{equation}
	From the definition of $\calA^{*}$, we have
	\[
		\langle \calA\circ\pi_{1}(x), \pi_{2}(y) \rangle_{G} = \langle \pi_{1}(x), \calA^{*}\circ\pi_{2}(y) \rangle_{M}
	\]
	for any $x\in\mathbb{R}^{r}$ and $y\in\mathbb{R}^{m}$, which can also be written as 
	\begin{equation*}
		(AW_{r}x)^{\top}GZy = (W_{r}x)^{\top}M(\calA^{*}Zy) \ \  \Leftrightarrow \ \ 
		x^{\top}W_{r}^{\top}A^{\top}GZy = x^{\top}W_{r}^{\top}M\calA^{*}Zy .
	\end{equation*}
	Thus, we have
	\begin{equation}\label{relation_At}
		W_{r}^{\top}A^{\top}GZ = W_{r}^{\top}M\calA^{*}Z
	\end{equation}
	Combining \eqref{mat_At} and \eqref{relation_At} and using $W_{r}^{\top}MW_{r}=I_{r}$, we finally obtain
	\begin{equation*}
		\rtA = W_{r}^{\top}MW_{r}\rtA = W_{r}^{\top}M\calA^{*}Z = W_{r}^{\top}A^{\top}GZ .
	\end{equation*}
	The proof is completed.
\end{proof}

The following result will be used throughout the paper.
\begin{lemma}\label{lem:pseudo}
	If $\calR(W_r)=\calR(M)$ and $W_{r}^{\top}MW_{r}=I_r$, then the Moor-Penrose pseudoinverse of $M$ can be expressed as 
	\[
		M^{\dag} = W_{r}W_{r}^{\top} . 
	\]
\end{lemma}
\begin{proof}
	Let $\bar{M}=W_{r}W_{r}^{\top}$. We only need to verify the following four identities:
	\begin{align*}
		& M\bar{M}M = M, \ \ \  (M\bar{M})^{\top} = M\bar{M}, \\
		& \bar{M}M\bar{M} = \bar{M}, \ \ \ (\bar{M}M)^{\top} = \bar{M}M .
	\end{align*}
	The third identity is the easiest to verify: $\bar{M}M\bar{M}=W_{r}W_{r}^{\top}MW_{r}W_{r}^{\top}=W_{r}W_{r}^{\top}=\bar{M}$.
	Suppose the compact-form eigenvalue decomposition of $M$ is
	\[
		M = P_{r}\Lambda_{r}P_{r}^{\top} , \ \ \ 
		\Lambda_{r} = \diag(\lambda_{1},\dots,\lambda_{r}) ,
	\]
	where $\lambda_{1}\geq\cdots\geq\lambda_{r}>0$ and $P_{r}\in\mathbb{R}^{n\times r}$ with $2$-orthonormal columns. Since $\calR(W_r)=\calR(M)=\calR(P_r)$, there exist $D\in\mathbb{R}^{r\times r}$ such that $W_{r}=P_{r}D$. It follows that 
	$
		I_{r}=W_{r}^{\top}MW_{r}=D^{\top}\Lambda_{r}D .
	$
	Therefore, it follows that
	\[
		M\bar{M}M  = MP_{r}DD^{\top}P_{r}^{\top}M = P_{r}\Lambda_{r}DD^{\top}\Lambda_{r}P_{r}^{\top} 
		= P_{r}\Lambda_{r}P_{r}^{\top} = M,
	\]
	since $\Lambda_{r}DD^{\top}=D^{\top}\Lambda_{r}D=I_{r}$. Similarly, we have
	\begin{align*}
		& M\bar{M} = P_{r}\Lambda_{r}P_{r}^{\top}P_{r}DD^{\top}P_{r}^{\top} 
		= P_{r}\Lambda_{r}DD^{\top}P_{r}^{\top} = P_{r}P_{r}^{\top} = (M\bar{M})^{\top}, \\
		& \bar{M}M = P_{r}DD^{\top}P_{r}^{\top}M
		= P_{r}DD^{\top}\Lambda_{r}P_{r}^{\top} 
		= P_{r}P_{r}^{\top} = (\bar{M}M)^{\top} . 
	\end{align*}
	Now all the four identities have been verified.
\end{proof}

\subsection{Characterizing GSVD by singular value expansion}\label{sec3.2}
In this subsection, we consider the simpler case that $G=I_{m}$, since it has direct connections with the GSVD of a matrix pair. For notational simplicity, let $\mathcal{X}=(\calR(M),\langle\cdot,\cdot\rangle_{M})$ and $\mathcal{Y}=(\mathbb{R}^{m},\langle\cdot,\cdot\rangle_{2})$. For the linear compact operator 
\begin{equation}\label{linear_op}
	\calA: \calX \rightarrow \calY, \ \ \ 
	v \mapsto Av
\end{equation}
between the two Hilbert spaces $\calX$ and $\calY$, where $v$ and $Av$ are column vectors under the canonical bases of $\mathbb{R}^{n}$ and $\mathbb{R}^{m}$, it has the singular value expansion (SVE) with finite terms; see e.g. \cite[\S 15.4]{Kress2014}. Here we use the terminology ``SVE" instead of ``SVD" to distinguish it from the SVD of a matrix. The theory of SVE for $\calA$ states that there exist positive scalars $\sigma_{1}\geq\cdots\geq\sigma_{d}>0$, two orthonormal systems $\{f_{i}\}_{i=1}^{d}\subseteq \calX$ and $\{h_{i}\}_{i=1}^{d}\subseteq \calY$ such that
\begin{equation}\label{SVE1}
	\calA f_{i} = \sigma_{i}h_i, \ \ \ 
	\calA^{*}h_{i} = \sigma_{i}f_{i} ,
\end{equation}
and any $v\in\calX$ has the expansion
\begin{equation}\label{SVE2}
	v = v_{0} + \sum_{i=1}^{d}\langle v, f_{i} \rangle_{M}f_{i}
\end{equation}
with some $v_{0}\in\mathrm{ker}(\calA)$, and
\begin{equation}\label{SVE3}
	\calA v = \sum_{i=1}^{d}\sigma_{i}\langle v, f_{i} \rangle_{M}h_{i},
\end{equation}
where $d=\mathrm{dim}(\mathrm{im}(\calA))$. Here we use $\mathrm{ker}(\cdot)$ and $\mathrm{im}(\cdot)$ to denote the kernel and image of a linear operator, respectively, to distinguish them from the null space $\calN(\cdot)$ and range space $\calR(\cdot)$ of a matrix. 

The following result provides more details about the SVE of $\calA$.
\begin{theorem}\label{thm:sve}
	For any $A\in\mathbb{R}^{m\times n}$ and symmetric positive semidefinite matrix $M\in\mathbb{R}^{n\times n}$ with rank $r$, define the linear operator $\calA$ as \eqref{linear_op}. Then there exist an $M$-orthonormal matrix $F\in\mathbb{R}^{n\times r}$, a $2$-orthonormal matrix $H\in\mathbb{R}^{m\times m}$ and a diagonal matrix $\Sigma\in\mathbb{R}^{m\times r}$, such that for any $v\in\calX$ and $u\in\calY$, it holds that
	\begin{equation}\label{SVE4}
		\calA v = H\Sigma F^{\top}Mv, \ \ \ 
		\calA^{*} u = F\Sigma^{\top} H^{\top}u 
	\end{equation}
	under the canonical bases of $\mathbb{R}^{n}$ and $\mathbb{R}^{m}$.
\end{theorem}
\begin{proof}
	Let $\calX_{1}=\mathrm{span}\{f_{i}\}_{i=1}^{d}$. We first prove $\calX=\mathrm{ker}(\calA)\oplus \calX_{1}$. Noticing \eqref{SVE2}, we only need to prove $\mathrm{ker}(\calA)\cap \calX_{1}=\{\mathbf{0}\}$. Let $v = \sum_{j=1}^{d}\mu_{j}f_{j}\in\mathrm{ker}(\calA)\cap \calX_{1}$. By \eqref{SVE3}, it follows $\mathbf{0}=\calA v=\sum_{j=1}^{n}\sigma_i\mu_{i}h_{i}$, leading to $\sigma_i\mu_i=0$ for $1\leq i\leq d$. Since $\sigma_{i}>0$, we have $\mu_{i}=0$ for $1\leq i\leq d$, thereby $v=\mathbf{0}$. Then we prove $\mathrm{ker}(\calA)\perp_{M}\calX_{1}$, where $\perp_{M}$ is the orthogonal relation in $\calX$. For any $v\in\mathrm{ker}(\calA)$ and any $f_i$, by \eqref{SVE1} we have $f_{i}=\sigma_{i}^{-1}\calA^{*}h_{i}$, thereby
	\begin{equation*}
		\langle v, f_{i}\rangle_{M} = \langle v, \sigma_{i}^{-1}\calA^{*}h_{i}\rangle_{M} =
		\langle \calA v, \sigma_{i}^{-1}h_{i}\rangle_{2} = \langle \mathbf{0}, \sigma_{i}^{-1}h_{i}\rangle_{2} = 0 .
	\end{equation*}
	Note $\mathrm{dim}(\calX_1)=d$. Therefore, we can find $r-d$ $M$-orthonormal vectors in $\mathrm{ker}(\calA)$ that are $M$-orthogonal to each $f_i$. Denote these vectors by $\{f_{d+1},\dots,f_{r}\}$. Then $\{f_{i}\}_{i=1}^{r}$ constitute a complete orthonormal basis for $\calX$.
	From \eqref{SVE3} we have $\mathrm{im}(\calA)=\mathrm{span}\{h_{i}\}_{i=1}^{d}=:\calY_1$. Using the relation $\mathrm{ker}(\calA^{*})=\mathrm{im}(\calA)^{\perp}=\calY_{1}^{\perp}$, where the orthogonality is taken in $(\mathbb{R}^{m},\langle\cdot,\cdot\rangle_{2})$, there exist $m-d$ 2-orthonormal $\{h_{d+1},\dots,h_{m}\}\subseteq \mathrm{ker}(\calA^{*})$ such that $\{h_{i}\}_{i=1}^{m}$ constitute a complete orthonormal basis for $\calY$.
	
	Therefore, for any $v\in\calX$, it can be written as $v=\sum_{i=1}^{r}\langle v, f_{i}\rangle_{M}f_{i}$, and thereby
	\begin{equation*}
		\calA v=\sum_{i=1}^{r}\langle v, f_{i}\rangle_{M}\calA f_{i} = \sum_{i=1}^{r}\sigma_{i}\langle v, f_{i} \rangle_{M}h_{i} ,
	\end{equation*}
	where we define $\sigma_{d+1}=\cdots=\sigma_{r}=0$. Similarly, for any $u\in\calY$ with the expansion $u=\sum_{i=1}^{m}\langle u, h_{i}\rangle_{2}h_{i}$, it holds
	\begin{equation*}
		\calA^{*} u = \sum_{i=1}^{m}\langle u, h_{i}\rangle_{M}\calA^{*}h_{i} = \sum_{i=1}^{r}\sigma_{i}\langle u, h_{i} \rangle_{2}f_{i}.
	\end{equation*}
	Let the matrices $F=(f_{1},\dots,f_{r})$, $H=(h_{1},\dots,h_{m})$ and $\Sigma=\begin{pmatrix}
		\Sigma_{d} & \\ & \mathbf{0}\end{pmatrix}\in\mathbb{R}^{m\times r}$ with $\Sigma_{d}=\diag(\sigma_{1},\dots,\sigma_{d})$. Then \eqref{SVE4} is just the matrix-form of the above two identities.
\end{proof}

One can verify that \eqref{SVE1}, \eqref{SVE2} and \eqref{SVE3} can be derived from \eqref{SVE4}. Therefore, the two relations in \eqref{SVE4} describe completely the SVE of $\calA$. In the following part, we use the notation 
\begin{equation}
	\calA \sim H\Sigma F^{\top}
\end{equation}
to denote the SVE of $\calA$. From the proof of \Cref{thm:sve}, we have the  following basic relations for the four important subspaces:
\begin{subequations}\label{basic_space}
	\begin{align}
		& \mathrm{ker}(\calA) = \mathrm{span}\{f_{i}\}_{i=d+1}^{r}, \ \ \ \ \ 
		\mathrm{im}(\calA) = \mathrm{span}\{h_{i}\}_{i=1}^{d}, \\
		& \mathrm{ker}(\calA^{*}) = \mathrm{span}\{h_{i}\}_{i=d+1}^{m}, \ \ \ 
		\mathrm{im}(\calA^{*}) = \mathrm{span}\{f_{i}\}_{i=1}^{d} .
	\end{align}
\end{subequations}
From the theory of SVE for linear compact operators, if the multiplicity of $\sigma_{i}$ is $1$, then the corresponding $f_i$ and $h_i$ are uniquely determined (at most differing by a sign). If the multiplicity of $\sigma_{i}$ is $m_i>1$, then there are $m_i$ corresponding linear independent $\{f_i\}$ and $\{h_i\}$, respectively, which are $M$-orthonormal and $2$-orthonormal.

Based on \Cref{thm:sve}, now we can use the SVE of $\calA$ to characterize the GSVD of $\{A, L\}$. Remember that we consider those $x_i \in \calR(M)$ for $1\leq i\leq r$.
\begin{theorem}\label{thm:sve_gsvd1}
	Let the GSVD of $\{A,L\}$ be \eqref{GSVD1} and let $\calA$ be defined as \eqref{linear_op} with $M=A^{\top}A+L^{\top}L$. Partition $P_A$ and $X$  as
	\begin{equation}
		P_{A} =
		\bordermatrix*[()]{%
			P_{A1} & P_{A2} & P_{A3} & m \cr
			q_1 & q_2 & m-q_1-q_2   \cr
		} \ , \ \ \
		X =
		\bordermatrix*[()]{%
			X_{1} & X_{2} & X_{3} & X_{4} & n \cr
			q_1 & q_2 & q_3 & n-r \cr
		} 
	\end{equation}
	and let $\widetilde{X}_{1}=(X_{1} \ X_{2} \ X_{3})$. Then the SVE of $\calA$ has the form
	\begin{equation}
		\calA \sim P_{A}\Sigma_{A}\widetilde{X}_{1}^{\top} ,
	\end{equation}
	and it holds that
	\begin{subequations}\label{q11}
		\begin{align}
			& \mathrm{ker}(\calA) = \calR(X_{3}), \ \ \ \ \ \ \ \ \ \ 
			\mathrm{im}(\calA) = \calR((P_{A1}\ P_{A2})),  \\
			& \mathrm{ker}(\calA^{*}) = \calR(P_{A3}), \ \ \ \ \ \ \  
			\mathrm{im}(\calA^{*}) = \calR((X_{1}\ X_{2})) .
		\end{align}
	\end{subequations}
\end{theorem}
\begin{proof}
	Using the GSVD of $\{A,L\}$, we have 
	\begin{align*}
		A^{\top}A + L^{\top}L
		= X^{-\top}\left(\begin{pmatrix}
			\Sigma_{A}^{\top}\Sigma_{A} &  \\
			& \mathbf{0}
		\end{pmatrix} + \begin{pmatrix}
			\Sigma_{L}^{\top}\Sigma_{L} &  \\
			& \mathbf{0}
		\end{pmatrix}\right)X^{-1}
		= X^{-\top}\begin{pmatrix}
			I_{r} &  \\
			& \mathbf{0}
		\end{pmatrix}X^{-1},
	\end{align*}
	which leads to $\mathrm{rank}(M)=r$ and
	\begin{align*}
		\begin{pmatrix}
			I_{r} &  \\
			& \mathbf{0}
		\end{pmatrix} = 
		\begin{pmatrix}
			\widetilde{X}_{1}^{\top} \\ X_{4}^{\top}
		\end{pmatrix} M 
		\begin{pmatrix}
			\widetilde{X}_{1} & X_{4}
		\end{pmatrix} = 
		\begin{pmatrix}
			\widetilde{X}_{1}^{\top}M\widetilde{X}_{1} & \widetilde{X}_{1}^{\top}MX_{4} \\
			X_{4}^{\top}M\widetilde{X}_{1} & X_{4}^{\top}MX_{4}
		\end{pmatrix} .
	\end{align*}
	Therefore, we have $\widetilde{X}_{1}^{\top}M\widetilde{X}_{1}=I_r$. Note that $\calR(\widetilde{X}_1)\subseteq\calR(M)$. It follows that $\widetilde{X}_1$ is an $M$-orthonormal basis of $(\calR(M),\langle\cdot,\cdot\rangle_{M})$, thereby we obtain from \eqref{lem:pseudo} that $M^{\dag}=\widetilde{X}_{1}\widetilde{X}_{1}^{\top}$. Notice from \eqref{GSVD1} that 
	\[A(\widetilde{X}_1\ X_{4}) = P_{A}(\Sigma_{A}\ \mathbf{0}) \Rightarrow A\widetilde{X}_1 = P_{A}\Sigma_{A} .\]
	Thus, we have $AM^{\dag} = A\widetilde{X}_{1}\widetilde{X}_{1}^{\top} = P_{A}\Sigma_{A}\widetilde{X}_{1}^{\top}$.
	For any $v\in (\calR(M),\langle\cdot,\cdot\rangle_{M})$, it holds
	\begin{align*}
		\calA v = \calA\calP_{\calR(M)}v = AM^{\dag}Mv = P_{A}\Sigma_{A}\widetilde{X}_{1}^{\top}Mv .
	\end{align*}
	Using the commutative diagram \eqref{commu_dig2} and noticing $G=Z=I_{m}$ for the current case, we have for any $u\in (\mathbb{R}^{m},\langle\cdot,\cdot\rangle_{2})$ that 
	\begin{equation}\label{adjoint}
		\calA^{*}u = \pi_{1}\circ\rtA(u) = \widetilde{X}_{1}(\widetilde{X}_{1}^{\top}A^{\top})u = M^{\dag}A^{\top}u 
		= (AM^{\dag})^{\top}u = \widetilde{X}_{1}\Sigma_{A}^{\top}P_{A}^{\top}u ,
	\end{equation}
	where we have used $\rtA = \widetilde{X}_{1}^{\top}A^{\top}$ by \Cref{lem:matrix_op}. This proves that the SVE of $\calA$ has the form $P_{A}\Sigma_{A}\widetilde{X}_{1}^{\top}$. 
	
	From the SVE of $\calA$ we have $\mathrm{dim}(\mathrm{im}(\calA))=q_1+q_2$. Using the relations \eqref{basic_space}, it follows that $\mathrm{im}(\calA) = \calR((P_{A1}\ P_{A2}))$. Since $P_{A}$ is a 2-orthogonal matrix, we then have $\mathrm{ker}(\calA^{*}) = \calR(P_{A3})$. The other two relations can also be verified easily.
\end{proof}

Corresponding to \Cref{thm:sve_gsvd1}, we have the following result.
\begin{theorem}\label{thm:sve_gsvd2}
	Define $\calL$ as
	\begin{equation}
		\calL: (\calR(M),\langle\cdot,\cdot\rangle_{M}) \rightarrow (\mathbb{R}^{p},\langle\cdot,\cdot\rangle_{2}), \ \ \ 
		v \mapsto Lv,
	\end{equation}
	where $v$ and $Lv$ are column vectors under the canonical bases of $\mathbb{R}^{n}$ and $\mathbb{R}^{p}$. Partition $P_{L}$ as 
	\begin{equation}
		P_{L} =
		\bordermatrix*[()]{%
			P_{L1} & P_{L2} & P_{L3} & p \cr
			p-q_2-q_3 & q_2 & q_3  \cr
		} .
	\end{equation}
	Then the SVE of $\calL$ has the form
	\begin{equation}
		\calL \sim P_{L}\Sigma_{L}\widetilde{X}_{1}^{\top},
	\end{equation}
	and it holds that
	\begin{subequations}\label{q22}
		\begin{align}
			& \mathrm{ker}(\calL) = \calR(X_{1}), \ \ \ \ \ \ \ \ \ \ 
			\mathrm{im}(\calL) = \calR((P_{L2}\ P_{L3})),  \\
			& \mathrm{ker}(\calL^{*}) = \calR(P_{L1}), \ \ \ \ \ \ \ \ 
			\mathrm{im}(\calL^{*}) = \calR((X_{2}\ X_{3})) .
		\end{align}
	\end{subequations}
\end{theorem}

\Cref{lem:gsvd_new} together with \Cref{thm:sve_gsvd1} and \Cref{thm:sve_gsvd2} characterizes completely the GSVD of $\{A,L\}$ based on the SVEs of linear operators $\calA$ and $\calL$. Particularly, they show that the nontrivial part $\widetilde{X}_1$ is the common right SVE components of $\calA$ and $\calL$, while $P_{A}$ and $P_{L}$ are the left SVE components of $\calA$ and $\calL$, respectively. Moreover, the relations \eqref{q11} and \eqref{q22} use the image spaces and kernel spaces of $\calA, \ \calA^{*}$ and $\calL, \ \calL^{*}$ to describe the structure of each GSVD blocks and give a new explanation of the three numbers $q_1$, $q_2$ and $q_3$ in \eqref{GSVD1}. 

Based on the SVE characterization of GSVD, we can expect to modify those algorithms for large-scale SVD computation for large-scale GSVD computation. To compute nontrivial extreme GSVD components, we generalize the standard GKB process from the viewpoint of linear operators.

\section{Generalizing the Golub-Kahan bidiagonalization}\label{sec4}
In this section, the generalization of GKB is quite general without requiring $M=A^{\top}A+L^{\top}L$, and we follow the notations and assumptions in \Cref{sec3.1}. For the linear operator in \eqref{Amap}, the iterative process of GKB can be described as follows; see \cite{caruso2019convergence} for discussions about GKB for linear compact operators. Choosing a nonzero vector $b\in (\mathbb{R}^{m},\langle\cdot,\cdot\rangle_{G})$, the basis recursive relations are 
\begin{equation}\label{GKB_1}
	\begin{cases}
		\beta_1 u_1 = b,  \\
		\alpha_{i}v_i = \calA^{*}u_i -\beta_i v_{i-1},  \\
		\beta_{i+1}u_{i+1} = \calA v_{i} - \alpha_i u_i, 
	\end{cases}
\end{equation}
where $u_{i}\in (\mathbb{R}^{m},\langle\cdot,\cdot\rangle_{G})$ and $v_i\in(\calR(M),\langle\cdot,\cdot\rangle_{M})$, and $\alpha_i$ and $\beta_{i}$ are positive scalars such that $\|v_i\|_{M}=\|u_{i}\|_{G}=1$. Note that $v_{0}:=\mathbf{0}$ for the initial step.

For the purpose of practical computation, we need to derive a matrix-form expression of the recursive relations. Using the isomorphisms $\pi_1$ and $\pi_2$, denote $u_i$ and $v_i$ by $u_i=Zy_{i}$ and $v_i=W_{r}x_i$ with $y_i\in\mathbb{R}^{m}$ and $x_i\in\mathbb{R}^{r}$ for any $i\geq 1$. Then we have
\begin{align*}
	& \calA v_{i} = \calA\circ\pi_{1}(x_{i}) = \pi_{2}\circ\rA x_{i} = Z Z^{-1}AW_{r}x_i = Av_{i}, \\
	&  \calA^{*}u_i = \calA^{*}\circ\pi_{2}(y_{i}) = \pi_{1}\circ\rtA y_{i} = W_{r}W_{r}^{\top}A^{\top}GZy_{i} = M^{\dag}A^{\top}Gu_{i}.
\end{align*}
Therefore, the last two recursions in \eqref{GKB_1} can be written in the matrix-vector forms
\begin{equation}\label{GKB_2}
	\begin{cases}
		\alpha_{i}v_i = M^{\dag}A^{\top}Gu_i -\beta_i v_{i-1}   \\
		\beta_{i+1}u_{i+1} = Av_{i} - \alpha_i u_i .
	\end{cases}
\end{equation}
Using \eqref{GKB_2}, the GKB of $\calA$ can be proceeded step by step. We name the above iterative process the \textit{generalized Golub-Kahan bidiagonalization} (\textsf{gGKB}). The pseudocode of \textsf{gGKB} is shown in \Cref{alg:gGKB}.

\begin{remark1}
	If $G$ is also positive semidefinite, define the linear operator $\calA: (\calR(M),\langle\cdot,\cdot\rangle_{M}) \rightarrow (\calR(G),\langle\cdot,\cdot\rangle_{G})$ be $v\mapsto Av$. In this case, a similar \textsf{gGKB} process can be obtained. A slight difference is that the initial vector should satisfy $b \in \calR(G)$.
\end{remark1}

\begin{algorithm}[htb]
	\caption{Generalized Golub-Kahan bidiagonalization (\textsf{gGKB})}\label{alg:gGKB}
 	\algorithmicrequire \ $A\in\mathbb{R}^{m\times n}$, $M\in\mathbb{R}^{n\times n}$, $G\in\mathbb{R}^{m\times m}$, $b\in\mathbb{R}^{m}$
	\begin{algorithmic}[1]
		\State Initialize: let $\beta_1=\|b\|_G$, \ $u_1=b/\beta_1$, 
		\State Compute $\bar{s}=A^{\top}Gu_1$, \ $s=M^{\dag}\bar{s}$, 
		\State $\alpha_{1}=\|s\|_{M}$, \ $v_{1}=s/\alpha_{1}$  
		\For {$i=1,2,\dots,k,$}
		\State $q=Av_i-\alpha_iu_i$, 
		\State $\beta_{i+1}=\|q\|_G$, \ $u_{i+1}=q/\beta_{i+1}$; 
		\State $\bar{s}=A^{\top}Gu_{i+1}$, \ $s=M^{\dag}\bar{s}-\beta_{i+1}v_i$
		\State $\alpha_{i+1}= \|s\|_{M}$, \ $v_{i+1} = s/\alpha_{i+1}$
		\EndFor
	\end{algorithmic}
 	\algorithmicensure \ $\{\alpha_i, \beta_i\}_{i=1}^{k+1}$, \ $\{u_i, v_i\}_{i=1}^{k+1}$
\end{algorithm}

For large-scale matrices, we can not directly compute $M^{\dag}$. In this case, using the relation
\begin{equation}\label{ls_gGKB}
	M^{\dag}\bar{s} = \argmin_{s\in\mathbb{R}^{n}}\|Ms-\bar{s}\|_2,
\end{equation}
we compute $M^{\dag}\bar{s}$ by iteratively solving the above least squares problems. This can be done efficiently by using the LSQR algorithm \cite{Paige1982}. In this case, \textsf{gGKB} has the nested inner-outer iteration structure.

If $G=I_m$ and $M=I_n$, the \textsf{gGKB} becomes the standard GKB. If $G=I_m$ and $M$ is invertible, the \textsf{gGKB} is equivalent to the generalizations of GKB prosed in \cite{arioli2013generalized,chung2017generalized,chung2018efficient,li2023subspace,li2023preconditioned} with different forms. The following result describes the basic property of \textsf{gGKB}, very similar to that of the standard GKB.

\begin{proposition}\label{prop:gGKB1}
	For each $v_i$ generated by \textsf{gGKB}, it holds that $v_i\in\calR(M)$. The group of vectors $\{v_i\}_{i=1}^{k}$ is an $M$-orthonormal basis of the Krylov subspace 
	\begin{equation}\label{krylov1}
		\mathcal{K}_k(M^{\dag}A^{\top}GA, M^{\dag}A^{\top}Gb) = \mathrm{span}\{(M^{\dag}A^{\top}GA)^{i}M^{\dag}A^{\top}Gb\}_{i=0}^{k-1},
	\end{equation}
	and $\{u_i\}_{i=1}^{k}$ is a $G$-orthonormal basis of the Krylov subspace 
	\begin{equation}\label{krylov2}
		\mathcal{K}_k(AM^{\dag}A^{\top}G, b) = \mathrm{span}\{(AM^{\dag}A^{\top}G)^{i}b \}_{i=0}^{k-1}.
	\end{equation}
\end{proposition}
\begin{proof}
	To get a better understanding of \textsf{gGKB}, we give two proofs.

	\paragraph*{The first proof.}
	We prove $v_i\in\calR(M)$ by mathematical induction. First note $\calR(M^{\dag})=\calR(M)$ for the symmetric $M$. For $i=1$, we have $\alpha_1=M^{\dag}A^{\top}Gu_1\in\calR(M)$. Suppose $v_i\in\mathcal{R}(M)$ for $i\geq 1$. From the recursion \eqref{GKB_2} we obtain 
	\[ \alpha_{i+1}v_{i+1} = M^{\dag}A^{\top}Gu_i-\beta_{i+1}v_{i} \in \calR(M) , \]
	leading to $v_{i+1}\in\mathcal{R}(M)$. To prove the second property, we exploit the theory about the GKB of $\calA: (\calR(M),\langle\cdot,\cdot\rangle_{M}) \rightarrow (\mathbb{R}^{m},\langle\cdot,\cdot\rangle_{G})$ with starting vector $b$, which states that $\{v_i\}_{i=1}^{k}$  and $\{u_i\}_{i=1}^{k}$ are $M$-orthonormal basis and $G$-orthonormal basis of the two Krylov subspaces 
	\begin{align*}
		& \mathcal{K}_k(\calA^{*}\calA, \calA^{*}b) := \mathrm{span}\{(\calA^{*}\calA)^{i}\calA^{*}b\}_{i=0}^{k-1}, \\
		& \mathcal{K}_k(\calA\calA^{*}, b) := \mathrm{span}\{(\calA\calA^{*})^{i}b\}_{i=0}^{k-1}, 
	\end{align*}
	respectively. 
	For any $u\in\mathbb{R}^{m}=\pi_{2}(y)=Zy$, from the commutative diagram \eqref{commu_dig2} and using \Cref{lem:matrix_op} and \Cref{lem:pseudo}, we obtain
	\begin{align*}
		\calA^{*}u
		= \calA^{*}\circ\pi_{2}(y) = \pi_{1}\circ\rtA(y) 
		= W_{r} (W_{r}^{\top}A^{\top}GZ)y
		= M^{\dag}A^{\top}Gu .
	\end{align*}
	Therefore, we have
	\begin{equation*}
		(\calA^{*}\calA)^{i}\calA^{*}b = (M^{\dag}A^{\top}GA)^{i}M^{\dag}A^{\top}Gb , \ \ \ 
		(\calA\calA^{*})^{i}b = (AM^{\dag}A^{\top}G)^{i}b .
	\end{equation*}

	\paragraph*{The second proof.} 
	Using the coordinates of $u_i$ and $v_i$ under bases $W_{r}$ and $Z$, we can write the last two relations in \eqref{GKB_1} as 
	\begin{align*}
		\alpha_{i}W_{r}x_{i} = \calA^{*}Zy_{i}-\beta_{i}W_{r}x_{i-1},   \ \ \
		 \beta_{i+1}Zy_{i+1} = \calA W_{r}x_{i} - \alpha_i Zy_i ,
	\end{align*}
	where $v_{i}=W_{r}x_i$ and $u_{i}=Zy_{i}$. 
	Note that $Z^{\top}GZ=I_{m}$ implies $Z^{-1}=Z^{\top}G$. Letting $\bar{b}=Z^{-1}b$, multiplying from left the first and second relations by $W_{r}^{\top}M$ and $Z^{-1}$, and using \eqref{relation_At}, we obtain 
	\begin{equation}\label{GKB_3}
		\begin{cases}
			\beta_{1}y_{1} = \bar{b}_1  \\
			\alpha_{i}x_{i} = W_{r}^{\top}A^{\top}GZy_{i}-\beta_{i}x_{i-1} \\
			\beta_{i+1}y_{i+1} = Z^{\top}GAW_{r}x_{i} - \alpha_{i}y_{i} .
		\end{cases} 
	\end{equation}
	Since $G^{\top}=G$, if follows that \eqref{GKB_3} is the standard GKB of matrix $Z^{\top}GAW_{r}$ with starting vector $\bar{b}$. Therefore, $\{x_{i}\}_{i=1}^{k}$ and $\{y_{i}\}_{i=1}^{k}$ are 2-orthonormal bases of the two Krylov subspaces
	\begin{align*}
		& \mathrm{span}\{((Z^{\top}GAW_{r})^{\top}Z^{\top}GAW_{r})^{i}(Z^{\top}GAW_{r})^{\top}\bar{b}\}_{i=0}^{k-1},  \\
		& \mathrm{span}\{(Z^{\top}GAW_{r}(Z^{\top}GAW_{r})^{\top})^{i}\bar{b}\}_{i=0}^{k-1} ,
	\end{align*}
	respectively. Note that 
	\begin{align*}
		& \ \ \ \ W_{r} ((Z^{\top}GAW_{r})^{\top}Z^{\top}GAW_{r})^{i}(Z^{\top}GAW_{r})^{\top}\bar{b} \\
		&= W_{r}(W_{r}^{\top}A^{\top}GZZ^{\top}GAW_{r})^{i}W_{r}^{\top}A^{\top}GZ\bar{b} \\
		&= W_{r}(W_{r}^{\top}A^{\top}GAW_{r})^{i}W_{r}^{\top}A^{\top}Gb \\
		&= (W_{r}W_{r}^{\top}A^{T}GA)^{i}W_{r}W_{r}^{\top}A^{\top}Gb \\
		&= (M^{\dag}A^{\top}GA)^{i}M^{\dag}A^{\top}Gb .
	\end{align*}
	We have $v_{i}=W_{r}x_i\in\calR(M)$ and obtain \eqref{krylov1}. Similarly, we can obtain \eqref{krylov2}.
\end{proof}

It is easy to verify that \eqref{GKB_3} is equivalent to \eqref{GKB_1}. Note from \eqref{lem:matrix_op} that $\rtA=\rA^{\top}$ since $Z^{-1}=Z^{\top}G$. Therefore, the matrix representations of $\calA$ and $\calA^{*}$ are $Z^{\top}GAW_{r}\in\mathbb{R}^{m\times r}$ and $(Z^{\top}GAW_{r})^{\top}$, respectively, which maps a coordinate vector from $\mathbb{R}^{r}$ to $\mathbb{R}^{m}$. In this sense, we can say that the recursive relation \eqref{GKB_3} is the coordinate representation for the \textsf{gGKB} of $\calA$ under bases $W_r$ and $Z$. 

From the first proof of \Cref{prop:gGKB1}, we have $s\in\calR(M)$. Therefore, for each $i\geq 1$, if $s=M^{\dag}A^{\top}Gu_{i}-\beta_{i}v_{i-1}\neq \boldsymbol{0}$, then $\alpha_{i}=\|s\|_{M}\neq 0$. This indicates that even if $M$ is not positive definite, the \textsf{gGKB} does not terminate as long as $s$ or $q=Av_i-\alpha_iu_i$ is nonzero.  Here ``terminate'' means that $\alpha_{i}$ or $\beta_i$ equals zero at the current step. Suppose \textsf{gGKB} does not terminate before the $k$-th iteration, i.e. $\alpha_{i}\beta_{i}\neq 0$ for $1\leq i \leq k$. Then the $k$-step \textsf{gGKB} process generates an $M$-orthonormal matrix $V_{k+1}=(v_1,\dots,v_{k+1})\in \mathbb{R}^{n\times (k+1)}$ and a $G$-orthonormal matrix $U_{k+1}=(u_1,\dots,u_{k+1})\in \mathbb{R}^{m\times (k+1)}$, satisfying the relations
\begin{subequations}{\label{eq:GKB_matForm}}
	\begin{align}
		& \beta_1U_{k+1}e_{1} = b, \label{GKB6} \\
		& AV_k = U_{k+1}B_k, \label{GKB7} \\
		&  M^{\dag}A^{\top}GU_{k+1} = V_kB_{k}^{T}+\alpha_{k+1}v_{k+1}e_{k+1}^\top , \label{GKB8}
	\end{align}
\end{subequations}
where $e_1$ and $e_{k+1}$ are the first and $(k+1)$-th columns of $I_{k+1}$, and 
\begin{equation}
	B_{k}
	=\begin{pmatrix}
		\alpha_{1} & & & \\
		\beta_{2} &\alpha_{2} & & \\
		&\beta_{3} &\ddots & \\
		& &\ddots &\alpha_{k} \\
		& & &\beta_{k+1}
		\end{pmatrix}\in  \mathbb{R}^{(k+1)\times k}
\end{equation}
has full column rank. Note that it may occurs that $\beta_{k+1}=0$, which means \textsf{gGKB} terminates just at the $k$-th step, and in this case $v_{k+1}=\mathbf{0}$. 

We emphasize that \textsf{gGKB} will eventually terminate at most $\min\{m,r\}$ steps, since the column rank of $U_{k}$ or $V_{k}$ can not exceed $\min\{m,r\}$. Using the GSVD of $\{A, L\}$, we can give a detailed description of the ``terminate step" of \textsf{gGKB}, defined as
\begin{equation}\label{gGKB_termi}
	k_t=\min\{k: \alpha_{k+1}\beta_{k+1}=0\}.
\end{equation}
For any closed subspace $\calG$ of a Hilbert space, denote by $\calP_{\calG}$ the projection operator onto $\calG$. We have the following result.

\begin{theorem}\label{thm:gGKB_terminate}
	Define the linear operator $\calA$ as \eqref{linear_op} with $M=A^{\top}A+L^{\top}L$, where the GSVD of $\{A,L\}$ is as \eqref{GSVD1}. Suppose there are $l$ distinct positive $c_i$ in $\Sigma_{A}$ with $l$ subspaces $\calG_1,\dots,\calG_l$ spanned by the corresponding $p_{A,i}$. Then $k_t$ equals to the number of nonzero elements in $\{\calP_{\calG_1}b,\dots,\calP_{\calG_l}b\}$.
\end{theorem}

The following lemma is needed to prove this theorem. 
\begin{lemma}\label{lem:deg}
	For any square matrix $C$ and a vector $v$, define the degree of $v$ with respect to $C$ as
	\begin{equation*}
		\mathrm{deg}_{C}(v) = \min\{k: \exists \ p\in\mathcal{P}_{k} \ \mathrm{s.t.} \ p(C)v=\mathbf{0}\} ,
	\end{equation*}
	where $\mathcal{P}_k$ is the set of all polynomials with degrees not bigger than $k$. Then we have 
	\begin{equation}
		\mathrm{deg}_{A M^{\dag}A^{\top}}(b) = 
		\mathrm{deg}_{M^{\dag}A^{\top}A}(M^{\dag}A^{\top}b)=l .
	\end{equation}
\end{lemma}
\begin{proof}
	First notice that $\mathrm{deg}_{C}(v)$ is nothing but the maximum rank of $\{C^{i}v\}_{i=0}^{k}$ with respect to $k\geq 0$. By \Cref{thm:sve_gsvd1}, we have $A M^{\dag}A^{\top}=A\widetilde{X}_{1}\widetilde{X}_{1}^{\top}A^{\top}$. Using the relation $A\widetilde{X}_{1}=P_{A}\Sigma_{A}$, we have $A M^{\dag}A^{\top}=P_{A}(\Sigma_{A}\Sigma_{A}^{\top})P_{A}^{\top}$, which is the eigenvalue decomposition of $A M^{\dag}A^{\top}$. Thus, the positive eigenvalues of $A M^{\dag}A^{\top}$ are $1, c_{q_1+1}^2,\dots,c_{q_1+q_2}^2$ with the corresponding eigenvectors be the columns of $(P_{A1}\ P_{A2})$, and the corresponding eigenspaces are subspaces $\calG_1,\dots,\calG_l$. Denote the $l$ distinct positive eigenvalues by $\lambda_1,\dots,\lambda_l$ and let $G_i$ be those matrices with $2$-orthonormal columns spanning $\calG_i$ for $1\leq i\leq l$. Then we can write the eigenvalue decomposition of $A M^{\dag}A^{\top}$ as $A M^{\dag}A^{\top}=\sum_{i=1}^{l}\lambda_{i}G_{i}G_{i}^{\top}$, and we have $\calP_{\calG_i}=G_{i}G_{i}^{\top}$. For each $j\geq 0$, it follows that
	\begin{equation*}
		w_j := (AM^{\dag}A^{\top})^{j}b = \sum_{i=1}^{l}(\lambda_{i}G_{i}G_{i}^{\top})^{j}b =
		\sum_{i=1}^{l}\lambda_{i}^{j}G_{i}G_{i}^{\top}b,
	\end{equation*}
	since $G_{i}$ are mutually $2$-orthogonal. Without loss of generality, suppose there are $s$ nonzero elements in $\{\calP_{\calG_1}b,\dots,\calP_{\calG_l}b\}$ and $g_i=\calP_{\mathcal{G}_i}b/\|\calP_{\mathcal{G}_i}b\|_{2}\neq \mathbf{0}$ for $1\leq i\leq s$.
	Then $\{g_{i}\}_{i=1}^{s}$ are mutually 2-orthogonal, and 
	\[w_j= \sum_{i=1}^{s}\lambda_{i}^{j}g_{i}\|\calP_{\mathcal{G}_i}b\|_{2}=\sum_{i=1}^{s}\lambda_{i}^{j}g_{i}(g_{i}^{\top}b),\]
	since 
	\[g_{i}^{\top}b=(G_{i}G_{i}^{\top}b)^{\top}b/\|\calP_{\mathcal{G}_i}b\|_{2} 
	=\|G_{i}^{\top}b\|_{2}^2/\|G_{i}^{\top}b\|_2 = \|\calP_{\mathcal{G}_i}b\|_{2} . \]
	Thus, the rank of $\{w_j\}_{j=0}^{k}$ is at most $s$, leading to $ \mathrm{deg}_{AM^{\dag}A^{\top}}(b)\leq s$. On the other hand, for $1\leq k\leq s$, by setting $\bar{w}_j:=g_{j}(g_{j}^{\top}b)$, we have $(w_1\dots,w_k)=(\bar{w}_1,\dots,\bar{w}_s)T_k$, where 
	\begin{equation*}
		T_k = \begin{pmatrix}
			1 & \lambda_{1} & \cdots & \lambda_{1}^{k-1} \\
			1 & \lambda_{2} & \cdots & \lambda_{2}^{k-1} \\
			\vdots & \vdots & \cdots & \vdots \\
			1 & \lambda_{s} & \cdots & \lambda_{s}^{k-1}
		\end{pmatrix} 
		=: \bordermatrix*[()]{%
			T_{k1}  &  k \cr
			T_{k2}  &  \ \ s-k \cr
			k
		} \ .
	\end{equation*}
	Since $\lambda_{i}\neq\lambda_{j}$ for $1\leq i\neq j\leq k$, the Vandermonde matrix $T_{k1}$ is invertible, thereby $T_k$ has full column rank. Therefore, the rank of $\{w_i\}_{i=1}^{k}$ is $k$ for $1\leq k\leq s$, leading to $\mathrm{deg}_{AM^{\dag}A^{\top}}(b)\geq s$. This proves $\mathrm{deg}_{AM^{\dag}A^{\top}}(b)=s$. 
	
	To prove $\mathrm{deg}_{M^{\dag}A^{\top}A}(M^{\dag}A^{\top}b)=s$, it is sufficient to show 
	\[ \mathrm{rank}\left(\{(M^{\dag}A^{\top}A)^{j}M^{\dag}A^{\top}b\}_{i=0}^{k}\right)=\mathrm{rank}\left(\{w_j\}_{j=0}^{k}\right) \]
	for any $k\geq 0$.
	Notice that 
	\begin{equation*}
		(M^{\dag}A^{\top}A)^{j}M^{\dag}A^{\top}b= M^{\dag}A^{\top}(AM^{\dag}A^{\top})^{j} b = M^{\dag}A^{\top}w_i
	\end{equation*}
	and
	\begin{equation*}
		M^{\dag}A^{\top} = \widetilde{X}_{1}(A\widetilde{X}_{1})^{\top} = 
		\widetilde{X}_{1}\Sigma_{A}^{\top}P_{A}^{\top}=(X_{1} \ X_{2}C_{q_2})(P_{A1} \ P_{A2})^{\top} .
	\end{equation*}
	Let $\tilde{w}_j=(M^{\dag}A^{\top}A)^{j}M^{\dag}A^{\top}b$. It follows that $\mathrm{rank}\left(\{\tilde{w}_j\}_{j=0}^{k}\right)\leq\mathrm{rank}\left(\{w_j\}_{j=0}^{k}\right) $. To prove the inverse inequality, suppose $\{w_j\}_{j=0}^{k}$ are independent. We only need to show $\{\tilde{w}_j\}_{j=0}^{k}$ are independent. If there exist real numbers $\mu_0,\dots,\mu_{k}$ such that $\sum_{j=0}^{k}\mu_{j}\tilde{w}_{j}=\mathbf{0}$. Then $M^{\dag}A^{\top}Wz=\mathbf{0}$, where $W=(w_{0},\dots,w_{k})$ has full column rank and $z=(\mu_{0},\dots,\mu_{k})^{\top}$. By the expression of $M^{\dag}A^{\top}$, it follows that $Wz\in\calN(M^{\dag}A^{\top})=\calR(P_{A3})$. On the other hand, from $AM^{\dag}A^{\top}=P_{A}(\Sigma_{A}\Sigma_{A}^{\top})P_{A}^{\top}$ we get $Wz\in\calR(W)\subseteq \calR(AM^{\dag}A^{\top})=\calR((P_{A1}\ P_{A2}))$. Since $\calR((P_{A1}\ P_{A2}))\cap \calR(P_{A3})=\{\mathbf{0}\}$, we obtain $Wz=\mathbf{0} \Rightarrow z=\mathbf{0}$, meaning that $\{\tilde{w}_j\}_{j=0}^{k}$ are independent. This completes the proof.
\end{proof}

\begin{proofof}{\Cref{thm:gGKB_terminate}}
	Suppose gGKB terminates at the $k_t$-th step.
	By \Cref{prop:gGKB1}, the rank of $\{u_i\}_{i=1}^{k_t}$ is $k_t$, implying $k_t\leq \mathrm{deg}_{A M^{\dag}A^{\top}}(b)=s$ by \Cref{lem:deg}. Then we show $k_t\geq s$. Notice from the relations \eqref{GKB_1} and \eqref{GKB_2} that
	\begin{align*}
		& \alpha_{1}\beta_{1}v_{1}=M^{\dag}A^{\top}b , \\
		& \alpha_{i+1}\beta_{i+1}v_{i+1} = M^{\dag}A^{\top}Av_{i}-(\alpha_{i}^{2}+\beta_{i+1}^{2})v_{i}-\alpha_{i}\beta_{i}v_{i-1}
	\end{align*}
	for $1\leq i\leq k_t$, where we have used 
	\begin{align*}
		M^{\dag}A^{\top}Av_{i} 
		&= \alpha_{i}M^{\dag}A^{\top}u_{i} + \beta_{i+1}M^{\dag}A^{\top}u_{i+1} \\
		&= \alpha_{i}(\alpha_{i}v_{i}+\beta_{i}v_{i+1}) + \beta_{i+1}(\alpha_{i+1}v_{i+1}+\beta_{i+1}v_{i}).
	\end{align*}
	Therefore, it follows that
	\begin{equation*}
		v_{i+1} = \frac{1}{\alpha_{i+1}\beta_{i+1}}\left(M^{\dag}A^{\top}Av_{i} - 
		(\alpha_{i}^2+\beta_{i+1}^2)v_{i}- \alpha_{i}\beta_{i}v_{i-1}\right)
	\end{equation*}
	for $1\leq i < k_t$. Combining with $v_1=\frac{1}{\alpha_1\beta_1}M^{\dag}A^{\top}b$, the above recursion leads to
	\begin{equation*}
		v_{k+1}=\sum_{i=0}^{k}\xi_i(M^{\dag}A^{\top}A)^{i}M^{\dag}A^{\top}b, \ \ \ 
		\xi_{k}=1/\Pi_{i=1}^{k+1}\alpha_i\beta_i\neq 0
	\end{equation*}
	for $1\leq k < k_t$. Since $\mathbf{0}=\alpha_{k_t+1}\beta_{k_t+1}v_{k_t+1}$ is a linear combination of $v_{k_t}$ and $v_{k_t-1}$ with nonzero coefficients, the above identity implies that $\alpha_{k_t+1}\beta_{k_t+1}v_{k_t+1}$ must be a linear combination of $\{(M^{\dag}A^{\top}A)^{i}M^{\dag}A^{\top}b\}_{i=0}^{k_t}$ with nonzero coefficients, thereby $\{(M^{\dag}A^{\top}A)^{i}M^{\dag}A^{\top}b\}_{i=0}^{k_t}$ is linear dependent. By \Cref{lem:deg}, it follows that $k_t\geq s$. Finally, we obtain $k_t=s$.
\end{proofof}

Just as the standard GKB can be employed to approximate extreme SVD components, we will utilize gGKB to approximate nontrivial extreme GSVD components.

\section{GSVD computation by generalized Golub-Kahan bidiagonalization}\label{sec5}
We first show that \textsf{gGKB} can be used to approximate the SVE components. Then we use \Cref{thm:sve_gsvd1} and \Cref{thm:sve_gsvd2} to relate these approximations to the nontrivial GSVD components.

\subsection{Computing nontrivial GSVD components by gGKB}\label{subsec4.1}
Suppose \textsf{gGKB} does not terminate before the $k$-th step. Then the compact-form SVD of $B_k$ can be written as
\begin{equation}\label{svd_B}
  B_k = Y_{k}\Theta_kH_{k}^{\top}, \ \ \ 
  \Theta_k = \mathrm{diag}\left(\theta_{1}^{(k)},\dots,\theta_{k}^{(k)}\right), \ \ \
  \theta_{i}^{(k)}>\dots>\theta_{k}^{(k)}>0 ,
\end{equation}
where $Y_{k}=\left(y_{1}^{(k)},\dots,y_{k}^{(k)}\right)\in\mathbb{R}^{(k+1)\times k}$ and $H_{k}=\left(h_{1}^{(k)},\dots,h_{k}^{(k)}\right)\in\mathbb{R}^{k\times k}$ are two 2-orthonormal matrices. The approximation to the SVE triplet $\left(c_{i}, p_{A,i}, x_i\right)$ of $\calA$ is defined as 
\begin{equation}
	\left(\bar{c}_{i}^{(k)},\bar{p}_{A,i}^{(k)},\bar{x}_{i}^{(k)}\right):=\left(\theta_{i}^{(k)}, U_{k+1}y_{i}^{(k)},V_{k}h_{i}^{(k)}\right).
\end{equation}
To measure the quality of this approximation, we give the following result.

\begin{theorem}\label{thm:sve_appr}
  The approximate SVE triplet for $\calA$ satisfies
  \begin{subequations}
	\begin{align}
		\calA\bar{x}_{i}^{(k)}-\bar{c}_{i}^{(k)}\bar{p}_{A,i}^{(k)} &= 0 , \\
		\calA^{*}\bar{p}_{A,i}^{(k)}- \bar{c}_{i}^{(k)}\bar{x}_{i}^{(k)} &= \alpha_{k+1}v_{k+1}e_{k+1}^{\top}y_{i}^{(k)} .
	  \end{align}
  \end{subequations}
\end{theorem}
\begin{proof}
  Note that $\calA v=Av$. The first relation can be verified using \eqref{GKB7}:
  \begin{equation*}
    \calA\bar{v}_{i}^{(k)}-\bar{\sigma}_{i}^{(k)}\bar{u}_{i}^{(k)} 
    = AV_{k}h_{i}^{(k)} - \theta_{i}^{(k)}U_{k+1}y_{i}^{(k)}
    = U_{k+1}\left(B_{k}h_{i}^{(k)}-\theta_{i}^{(k)}y_{i}^{(k)} \right) = 0 .
  \end{equation*}
  For the second relation, using \eqref{adjoint} that $\calA^{*}u=M^{\dag}A^{\top}u$, we obtain from \eqref{GKB8} that
  \begin{align*}
    \calA^{*}\bar{p}_{A,i}^{(k)}- \bar{c}_{i}^{(k)}\bar{x}_{i}^{(k)}
    &= M^{\dag}A^{\top}U_{k+1}y_{i}^{(k)}-\theta_{i}^{(k)}V_{k}h_{i}^{(k)} \\
    &= \left(V_{k}B_{k}^{\top}+\alpha_{k+1}v_{k+1}e_{k+1}^{\top}\right)y_{i}^{(k)} -\theta_{i}^{(k)}V_{k}h_{i}^{(k)} \\
    &= V_k(B_{k}^{\top}y_{i}^{(k)}-\theta_{i}^{(k)}h_{i}^{(k)}) + \alpha_{k+1}v_{k+1}e_{k+1}^{\top}y_{i}^{(k)} \\
    &= \alpha_{k+1}v_{k+1}e_{k+1}^{\top}y_{i}^{(k)} .
  \end{align*}  
  The proof is completed.
\end{proof}

Therefore, the triplet $\left(\bar{\sigma}_{i}^{(k)},\bar{p}_{A,i}^{(k)},\bar{x}_{i}^{(k)}\right)$ can be accepted as a satisfied SVE triplet at the iteration that $\left|\alpha_{k+1}v_{k+1}e_{k+1}^{\top}y_{i}^{(k)}\right|$ is sufficiently small. Using the connection between the SVE of $\calA$ and the GSVD of $\{A,L\}$ revealed by \Cref{thm:sve_gsvd1}, the tuple 
\begin{equation}
	\left(\bar{c}_{i}^{(k)},\bar{s}_{i}^{(k)},\bar{p}_{A,i}^{(k)},\bar{x}_{i}^{(k)}\right):=\left(\theta_{i}^{(k)},(1-(\theta_{i}^{(k)})^2)^{1/2},U_{k+1}y_{i}^{(k)},V_{k}h_{i}^{(k)}\right)
\end{equation}
can be used as a good approximation to a GSVD component. To further measure the quality of this approximation, note from \eqref{GSV} that 
\begin{equation}\label{GSV4}
	s_{i}^{2}A^{\top}Ax_i = c_{i}^{2}L^{\top}Lx_i , \ \ \ 1 \leq i\leq r .
\end{equation}
This is a well-known basic relation for GSVD, which indicates that the nontrivial generalized eigenvalues of the generalized eigenvalue problem $A^{T}Ax = \lambda L^{T}Lx$ are $\{\gamma_{i}^2\}_{i=1}^{r}$ and the corresponding generalized eigenvectors are $\{x_{i}\}_{i=1}^{r}$ \cite[\S 8.7]{Golub2013}. Now we can give the following result.

\begin{theorem}\label{coro:gsvd_appr}
	The above approximate GSVD tuple for $\{A, L\}$ satisfies
	\begin{equation}\label{res3}
		(\bar{s}_{i}^{(k)})^{2}A^{\top}A\bar{x}_{i}^{(k)} - (\bar{c}_{i}^{(k)})^{2}L^{\top}L\bar{x}_{i}^{(k)} = \alpha_{k+1}\beta_{k+1}Mv_{k+1}e_{k}^{\top}h_{i}^{(k)} .
	\end{equation}
\end{theorem}
\begin{proof}
	First notice from \eqref{GSV4} that 
	\begin{align*}
		& A^{\top}Ax_i = (c_{i}^2+s_{i}^2)A^{\top}Ax_i=c_{i}^{2}(A^{\top}A+L^{\top}L)x_{i}=c_{i}^{2}Mx_{i}, \\
		& L^{\top}Lx_i = (c_{i}^2+s_{i}^2)L^{\top}Lx_i=s_{i}^{2}(A^{\top}A+L^{\top}L)x_{i}=s_{i}^{2}Mx_{i}
	\end{align*}
	for $1\leq i\leq r$. Since $\widetilde{X}_{1}=(x_{1},\dots,x_{r})$ is an $M$-orthonormal basis of $(\calR(M),\langle\cdot,\cdot\rangle_{M})$, it follows that $A^{\top}A\calR(M)\subseteq\calR(M)$ and $L^{\top}L\calR(M)\subseteq\calR(M)$. Therefore, we have 
	$A^{\top}A\bar{x}_{i}^{(k)}, L^{\top}L\bar{x}_{i}^{(k)} \in \calR(M)$ due to $\bar{x}_{i}^{(k)}=V_{k}h_{i}^{(k)}\in\calR(M)$. By \Cref{thm:sve_appr}, we have 
	\begin{align*}
		& \ \ \ \ M^{\dag}[(\bar{s}_{i}^{(k)})^{2}A^{\top}A\bar{x}_{i}^{(k)} - (\bar{c}_{i}^{(k)})^{2}L^{\top}L\bar{x}_{i}^{(k)} ] 
		= M^{\dag}[A^{\top}A\bar{x}_{i}^{(k)} - (\bar{c}_{i}^{(k)})^{2}M\bar{x}_{i}^{(k)} ] \\
		&= \theta_{i}^{(k)}M^{\dag}A^{\top}U_{k+1}y_{i}^{(k)} - (\theta_{i}^{(k)})^2V_{k}h_{i}^{(k)} \\
		&= \theta_{i}^{(k)}\left(V_{k}B_{k}^{\top}+\alpha_{k+1}v_{k+1}e_{k+1}^{\top}\right)y_{i}^{(k)}-(\theta_{i}^{(k)})^2V_{k}h_{i}^{(k)} \\
		&= \alpha_{k+1}\beta_{k+1}v_{k+1}e_{k}^{\top}h_{i}^{(k)},
	\end{align*}
	where we have used $B_{k}^{\top}y_{i}^{(k)}=\theta_{i}^{(k)}h_{i}^{(k)}$ and $B_{k}h_{i}^{(k)}=\theta_{i}^{(k)}y_{i}^{(k)}$. Multiplying the above equality by $M$ and using $\calP_{\calR(M)}=MM^{\dag}$, we finally obtain \eqref{res3}.
\end{proof}

Combining \Cref{thm:sve_appr,coro:gsvd_appr}, it is more proper to use the residual norm 
\begin{equation}\label{res}
	\|r_{i}^{(k)}\|_2 := \left(\|A\bar{x}_{i}^{(k)}-\bar{c}_{i}^{(k)}\bar{p}_{A,i}^{(k)}\|_{2}^2+\|(\bar{s}_{i}^{(k)})^{2}A^{\top}A\bar{x}_{i}^{(k)} - (\bar{c}_{i}^{(k)})^{2}L^{\top}L\bar{x}_{i}^{(k)}\|_{2}^2 \right)^{1/2}
\end{equation}
to measure the quality of the approximate GSVD components of $A$. Since $\|v_{k+1}\|_{M}=1$, it follows from \eqref{res3} that
\begin{equation}
	\|r_{i}^{(k)}\|_2/\|(A^{\top}, L^{\top})^{\top}\|_2 \leq \alpha_{k+1}\beta_{k+1}|e_{k}^{\top}h_{i}^{(k)}|,
\end{equation}
because $\|M\|_{2}^{1/2}=\|(A^{\top}, L^{\top})^{\top}\|_2$. The easily computed quantity $\alpha_{k+1}\beta_{k+1}|e_{k}^{\top}h_{i}^{(k)}|$ is an upper bound of the scaling-invariant relative residual norm $\|r_{i}^{(k)}\|_2/\|(A^{\top}, L^{\top})^{\top}\|_2$, which can be used in a stopping criterion. 

We present the pseudocode of the \textsf{gGKB}-based GSVD computation (computing the GSVD components of $A$) in \Cref{alg:GSVD}. We remark that in order to approximate the GSVD components of $L$, the \textsf{gGKB} of $\calL$ should be used; the spirit is the same as that for $\calA$ and we omit it.

\begin{algorithm}[htb]
	\caption{The \textsf{gGKB}-based GSVD computation (\textsf{gGKB\_GSVD})}\label{alg:GSVD}
 	\algorithmicrequire \ $A\in\mathbb{R}^{m\times n}$, $L\in\mathbb{R}^{p\times n}$, $\textrm{tol}>0$
	\begin{algorithmic}[1]
		\State Initialize: choose a nonzero $b\in\mathbb{R}^{m}$; form $M=A^{\top}A+L^{\top}L$
		\State Compute $\beta_1, \ \alpha_1, \ u_1,\ v_1$ by \textsf{gGKB} 
		\For {$i=1,2,\dots,k,$}
		\State Compute $\beta_{k+1}, \ \alpha_{k+1}, \ u_{k+1},\ v_{k+1}$ by \textsf{gGKB}; form $B_{k}, \ U_{k+1}$ and $V_{k}$
		\State Compute the SVD of $B_k$ as \eqref{svd_B}
		\If {$\alpha_{k+1}\beta_{k+1}|e_{k}^{\top}h_{i}^{(k)}|<\mathrm{tol}$}
			\State Compute $\left(\bar{c}_{i}^{(k)},\bar{s}_{i}^{(k)},\bar{p}_{A,i}^{(k)},\bar{x}_{i}^{(k)}\right):=\left(\theta_{i}^{(k)},(1-(\theta_{i}^{(k)})^2)^{\frac{1}{2}},U_{k+1}y_{i}^{(k)},V_{k}h_{i}^{(k)}\right)$
		\EndIf
		\EndFor
	\end{algorithmic}
 	\algorithmicensure \ Approximate GSVD components $\left(\bar{c}_{i}^{(k)},\bar{s}_{i}^{(k)},\bar{p}_{A,i}^{(k)},\bar{x}_{i}^{(k)}\right)$
\end{algorithm}

\subsection{Convergence and accuracy}
We provide preliminary results about the convergence and accuracy of \textsf{gGKB\_GSVD} for GSVD computation. The following result demonstrates the good property of \textsf{gGKB\_GSVD} at the terminate step.

\begin{theorem}\label{thm:gsvd_exact}
	Following the notations and assumptions of \Cref{thm:gGKB_terminate}, then at the $k_t$-th step, the \textsf{gGKB\_GSVD} algorithm computes exactly $k_t$ GSVD components corresponding to the nonzero elements in $\{\calP_{\calG_1}b,\dots,\calP_{\calG_l}b\}$.
\end{theorem}
\begin{proof}
	By \Cref{thm:sve_appr,coro:gsvd_appr}, at the terminate step of \textsf{gGKB}, it computes exact the SVE components of $\calA$, which are also the exact GSVD components of $\{A, L\}$ by \Cref{thm:sve_gsvd1}. Following the notations in the proof of \Cref{lem:deg}, we need to prove that the $s$ vectors $\bar{p}_{A,i}^{(s)}$ belong separately to the invariant subspaces $\calG_1,\dots,\calG_s$. Since $\theta_{i}^{(s)}>0$ have different values and $\calG_i$ are mutually 2-orthogonal, these $\bar{p}_{A,i}^{(s)}$ must belong to different invariant subspaces. Therefore, we only need to prove $\calP_{\calG}\bar{p}_{A,i}^{(s)}=\bar{p}_{A,i}^{(s)}$ for each $1\leq i\leq s$, where $\calG=\calG_1\oplus\cdots\oplus\calG_s$. From the proof of \Cref{lem:deg}, we have $\bar{p}_{A,i}^{(s)}\in\mathcal{K}_{s}(AM^{\dag}A^{\top},b)=\mathrm{span}\{w_i\}_{i=0}^{s-1}$, and
	\begin{equation*}
		\widetilde{W}_s:=(w_{0},\dots,w_{s-1}) = \widetilde{G}_s\begin{pmatrix}
			g_{1}^{\top}b & & \\
			 & \ddots & \\
			 & & g_{1}^{\top}b
		\end{pmatrix}
		\begin{pmatrix}
			1 & \lambda_{1} & \cdots & \lambda_{1}^{s-1} \\
			1 & \lambda_{2} & \cdots & \lambda_{2}^{s-1} \\
			\vdots & \vdots & \cdots & \vdots \\
			1 & \lambda_{s} & \cdots & \lambda_{s}^{s-1}
			\end{pmatrix} =: \widetilde{G}_s \Lambda_s T_s ,
	\end{equation*} 
	where $\widetilde{G}_s=(g_{1},\dots,g_s)$. Since $g_{i}^{\top}b\neq0$ and $T_s$ is nonsingular, it follows that $\calR(\widetilde{W}_s)=\calR(\widetilde{G}_s)$. Thus, we can write $\bar{p}_{A,i}^{(s)}$ as $\bar{p}_{A,i}^{(s)}=\widetilde{G}_s z$ with a nonzero $z\in\mathbb{R}^{s}$. Now we immediately obtain
	\begin{equation*}
		\calP_{\calG}\bar{p}_{A,i}^{(s)} = \sum_{i=1}^{s}\calP_{\calG_i}\bar{p}_{A,i}^{(s)}
		= \sum_{i=1}^{s}G_{i}G_{i}^{\top}(\widetilde{G}_s z) = (\widetilde{G}_s\widetilde{G}_{s}^{\top})\widetilde{G}_s z 
		= \bar{p}_{A,i}^{(s)} ,
	\end{equation*} 
	which is the desired result.
\end{proof}

To investigate the convergence behavior of the approximations, we give the following result to describe the convergence speed of the Ritz values $\theta_{i}^{(k)}$.
\begin{theorem}\label{thm:conv}
	For any $1\leq i\leq q_1+q_2$, let
	\begin{equation}\label{ang}
		\theta_i = \arccos\frac{|c_{i}p_{A,i}^{\top}b|}{\|\Sigma_{A}^{\top}P_{A}^{\top}b\|_2} .
	\end{equation}
	Then at the $k$-th iteration of \textsf{gGKB\_GSVD}, it holds
	\begin{equation}\label{conv_er}
		0\leq c_{i}^2-(\theta_{i}^{(k)})^2 \leq (c_{1}^2-c_{r}^2)\left(\frac{\kappa_{i}^{(k)}\tan\theta_i}{C_{k-i}(1+2\gamma_i)} \right),
	\end{equation}
	where $C_j(t)$ is the $j$-th Chebyshev polynomial
	\begin{equation*}
		C_{j}(t) = \frac{1}{2}\left[(t+\sqrt{t^2-1})^k+(t+\sqrt{t^2-1})^{-k} \right], \ \ |t|\geq 1 
	\end{equation*}
	and 
	\[\gamma_i = \frac{c_{i}^2-c_{i+1}^2}{c_{i+1}^2-c_{r}^2}, \ \ \
	\kappa_{i}^{(k)} = \prod_{j=1}^{i-1}\frac{(\theta_{j}^{(k)})^2-c_{r}^2}{(\theta_{j}^{(k)})^2-c_{i}^2} \ (i>1), \ \ \
	\kappa_{1}^{(k)} = 1 \ (i=1) .\]
\end{theorem}
\begin{proof}
	Using the relations \eqref{GKB_3} with $G=Z=I_m$ and $W_{r}=\widetilde{X}_1$, it follows that the coordinate representation of \textsf{gGKB} is the standard GKB of $A\widetilde{X}_1$ with starting vector $b$. This GKB process is equivalent to the symmetric Lanczos process of $(A\widetilde{X}_1)^{\top}A\widetilde{X}_1\in\mathbb{R}^{r\times r}$ with starting vector $(A\widetilde{X}_1)^{\top}b$, which generates 2-orthogonal vectors $\{u_i\}_{i=1}^{k}$ and the symmetric tridiagonal matrix $B_{k}^{\top}B_{k}$; see e.g. \cite{larsen1998lanczos}. Since $A\widetilde{X}_1=P_{A}\Sigma_{A}$, it follows that $(A\widetilde{X}_1)^{\top}A\widetilde{X}_1=I_{r}(\Sigma_{A}^{\top}\Sigma_{A})I_{r}^{\top}$ is the eigenvalue decomposition of $(A\widetilde{X}_1)^{\top}A\widetilde{X}_1$, and $(A\widetilde{X}_1)^{\top}b=\Sigma_{A}^{\top}P_{A}^{\top}b$. Since the $i$-th eigenvector of $(A\widetilde{X}_1)^{\top}A\widetilde{X}_1$ is $e_i$ and
	\begin{equation*}
		(A\widetilde{X}_1)^{\top}b=\Sigma_{A}^{\top}P_{A}^{\top}b = 
		\begin{pmatrix}
			(P_{A1}^{\top}b)^{\top} & (P_{A2}^{\top}b)^{\top} & \mathbf{0}
		\end{pmatrix}^{\top} ,
	\end{equation*}
	the angle between $(A\widetilde{X}_1)^{\top}b$ and $e_i$ for $q_1+q_2+1\leq i\leq r$ is $\pi/2$, and for $1\leq i\leq q_1+q_2$ the angle is expressed as \eqref{ang}. Notice that the eigenvalues of $B_{k}^{\top}B_{k}$ are $(\theta_{i}^{(k)})^2$. Using the convergence theory of the symmetric Lanczos process (see e.g. \cite[Theorem 6.4]{saad2011numerical}), we immediately obtain \eqref{conv_er}.
\end{proof}

\Cref{thm:conv} indicates that the convergence rate of $\theta_{i}^{(k)}$ primarily depends on two factors: the clossness between $b$ and the corresponding vector $p_{A,i}$ and the degree of separation of $\theta_{i}^{(k)}$ from others. Therefore, usually we can expect rapid convergence to the extreme and well-separated positive $c_i$. Note again that the approximations will not converge to the GSVD components corresponding to those zero $c_{i}$, since the angle between $(A\widetilde{X}_1)^{\top}b$ and $e_i$ is $\pi/2$ for $q_1+q_2+1\leq i\leq r$.
The convergence behavior of $\bar{p}_{A,i}^{(k)}$ and $\bar{x}_{i}^{(k)}$ can also be described similarly based on the convergence theory of the symmetric Lanczos process, but the mathematical expressions are more complex. Interested readers can refer to \cite[\S 6.6]{saad2011numerical}

We remark that all the aforementioned results are derived for the \textsf{gGKB} with exact computations., i.e. we do not take into account rounding errors and computational errors arising from iteratively solving \eqref{ls_gGKB}. In the presence of rounding errors, the Lanczos-type iterative process behaves very differently from that in exact arithmetic. One well-known result is that the orthogonality of $u_i$ and $v_i$ will gradually lost, which leads to a delay of convergence of approximations and the appearance of spurious convergent quantities \cite{larsen1998lanczos}. Also, the inaccurate computation of $M^{\dag}\bar{s}$ may affect the final accuracy of the approximations. 
These issues for \textsf{gGKB\_GSVD} will be addressed in future work. We will demonstrate several of them in the following numerical experiments.

\section{Experimental results}\label{sec6}
We report some experimental results to demonstrate the performance of \textsf{gGKB\_GSVD} for computing nontrivial extreme GSVD components. All the experiments are performed in MATLAB R2023b using double precision. The codes are available at \url{https://github.com/Machealb/gsvd_iter}. For the starting vector of \textsf{gGKB} for $A$ and $L$, we use the random vector $b=\texttt{randn(m,1)}$ and $b=\texttt{randn(p,1)}$ with random seed $\texttt{rng(2024)}$, respectively.

\paragraph*{Example 1}
The matrix pair $\{A, L\}$ is constructed as follows. Set $m=n=p=1000$. Let $C_A=\mathrm{diag}(\{c_i\}_{i=1}^{n})$ with $c(1)=1, c(2)=0.95, c(3)=0.90$, $c(4:n-6)=\texttt{linspace(0.88,0.12,n-6)}$ and $c(n-2)=0.1, c(n-1)=0.05, c(n)=0.01$, and let $S_L=\mathrm{diag}(\{s_i\}_{i=1}^{n})$ with $s_i=(1-c_{i}^2)^{1/2}$. Then let $W$ be an orthogonal matrix by letting $W=\texttt{gallery(`orthog',n,2)}$, and $D=\texttt{diag(linspace(1,100,n))}$. Finally let $A=C_AW^{\top}D$ and $L=S_LW^{\top}D$. By the construction, $\{A,L\}$ is a regular matrix pair, and the $i$-th generalized singular value of $\{A, L\}$ are $c_i/s_i$, where the corresponding generalized singular vectors are the $i$-th columns of $I_n$, $I_n$ and $D^{-1}W$.

\begin{figure}[!htbp]
	\centering
	\subfloat
	{\label{fig:1a}\includegraphics[width=0.48\textwidth]{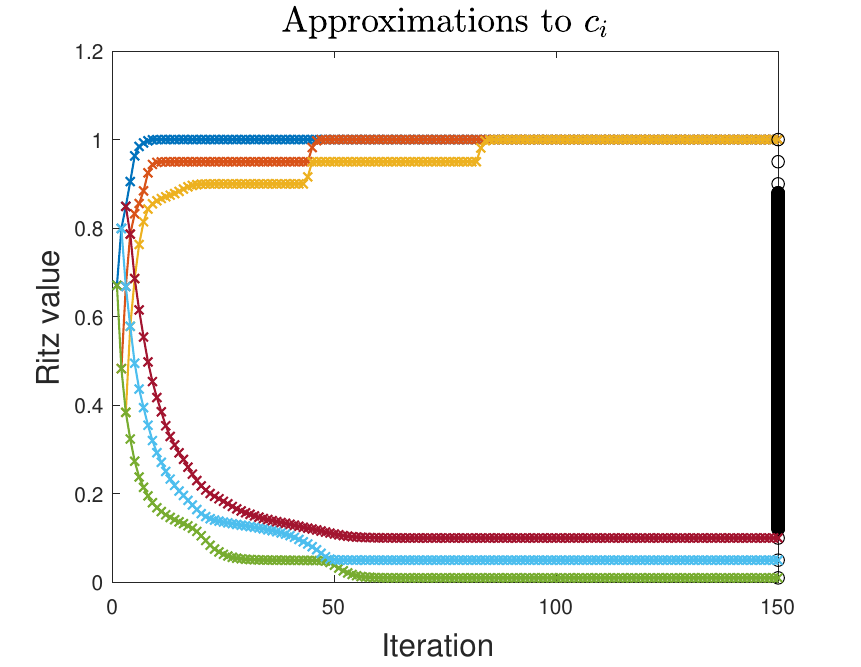}}\hspace{-0.0mm}
	\subfloat
	{\label{fig:1b}\includegraphics[width=0.48\textwidth]{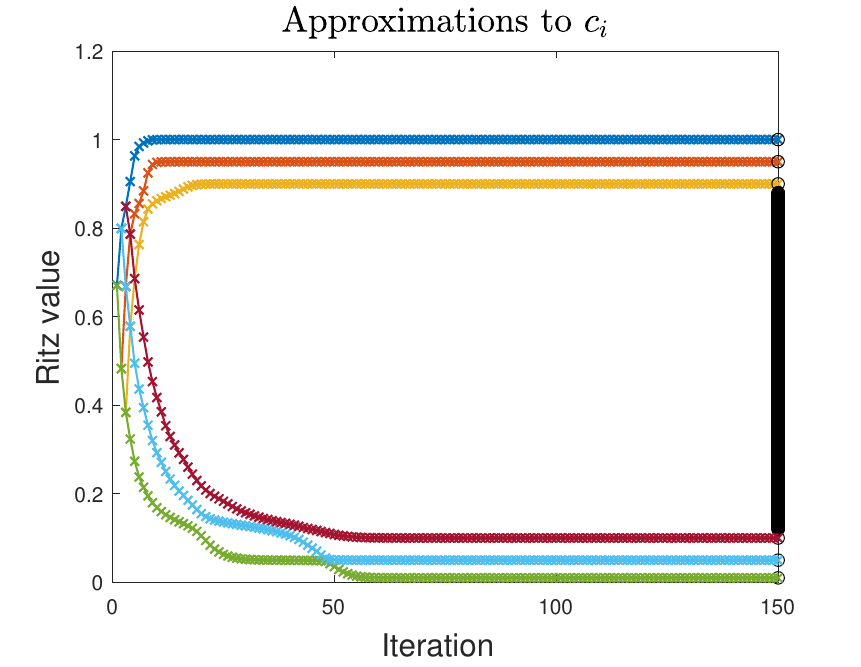}}
	\vspace{-2mm}
	\subfloat 
	{\label{fig:1c}\includegraphics[width=0.48\textwidth]{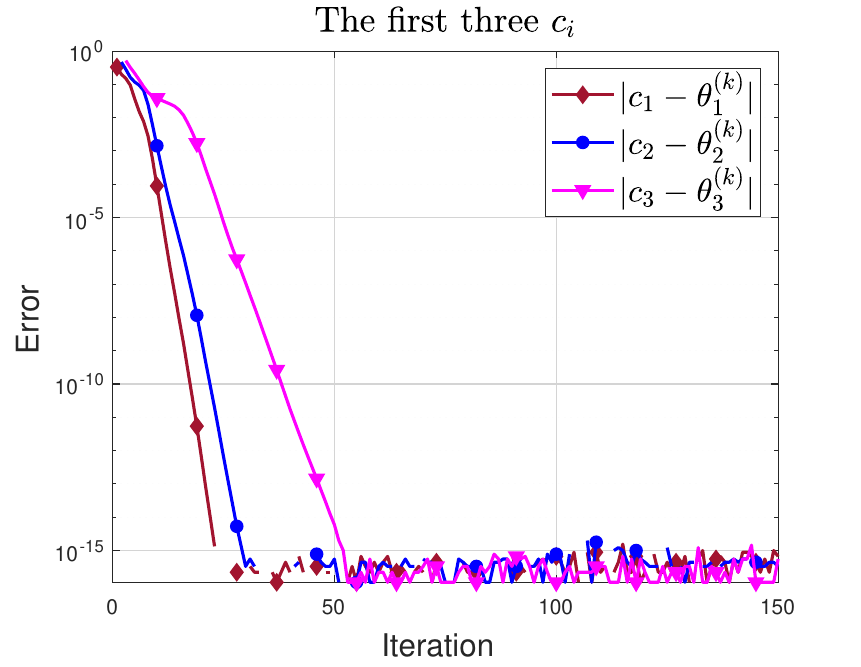}}\hspace{-0.0mm}
	\subfloat
	{\label{fig:1d}\includegraphics[width=0.50\textwidth]{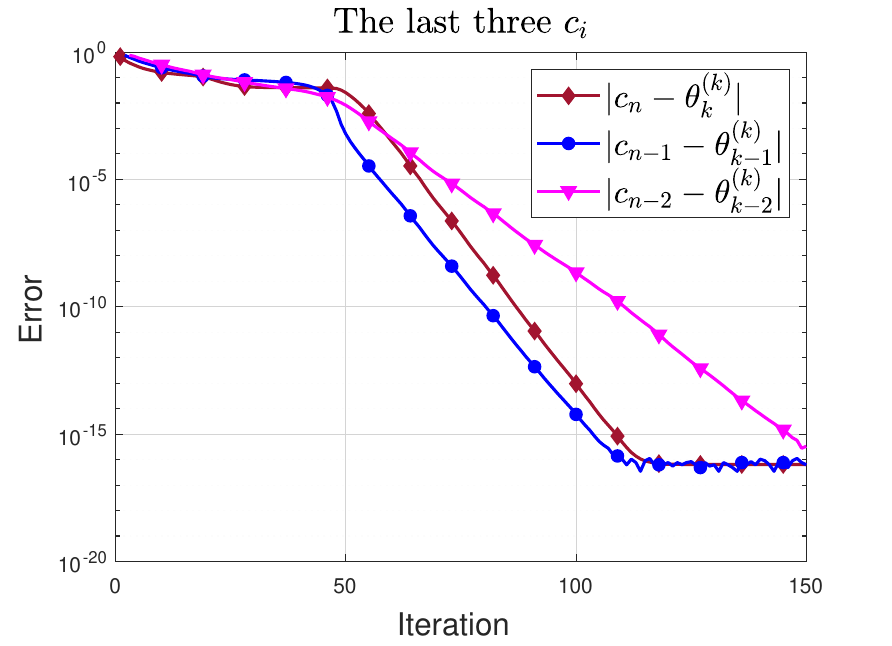}}
	\vspace{-2mm}
	\caption{Convergence and accuracy of approximations to $c_i$. Top: convergence of Ritz values $\theta_{i}^{(k)}$ to largest/smallest $c_i$ by \textsf{gGKB\_GSVD} without reorthogonalization (left) and with full reorthogonalization (right). Bottom: error curves (full reorthogonalization).}
	\label{fig1}
\end{figure}

In this experiment, we use \textsf{gGKB\_GSVD} to compute several largest and smallest generalized singular values, and show the convergence behavior of Ritz values $\theta_{i}^{(k)}$ where \textsf{gGKB} is performed with and without reorthogonalization. This is a small-scale problem, thereby we directly compute $M^{-1}$ for computing $M^{-1}\bar{s}$ at each iteration of \textsf{gGKB}. \Cref{fig1} shows the convergence of the first three largest and smallest Ritz values, where in the top two subfigures the right vertical lines indicate
the values of $c_i$. There are four findings. (1) If no reorthogonalization is used for \textsf{gGKB}, then as the iteration proceeds, the converged Ritz values may suddenly jump up (also may jump down for converging to those smallest $c_i$) to become a ghost and then converge to the next larger $c_i$; this phenomenon leads to the appearance of spurious copies of computed $c_i$. (2) If \textsf{gGKB} is performed with full reorthogonalization, the convergence of the Ritz values remains regular, and the first three largest and smallest Ritz values converge to the first three largest and smallest $c_i$, respectively. (3) The final accuracy of the approximated $c_i$ is around $\mathcal{O}(\mathbf{u})$, where $\mathbf{u}=2^{-53}\approx10^{-16}$ is the roundoff unit of double precision. (4) The convergence to those largest $c_i$ is faster than the convergence to those smallest $c_i$; also, the convergence to $c_1, c_2$ and $c_n, c_{n-1}$ is faster than that to $c_3$ and $c_{n-2}$. This is because $c_1, c_2$ are more well-separated from others than $c_3$; the same reason applies to $c_n, c_{n-1}$ and $c_{n-2}$.

\begin{figure}[!htbp]
	\centering
	\subfloat
	{\label{fig:2a}\includegraphics[width=0.48\textwidth]{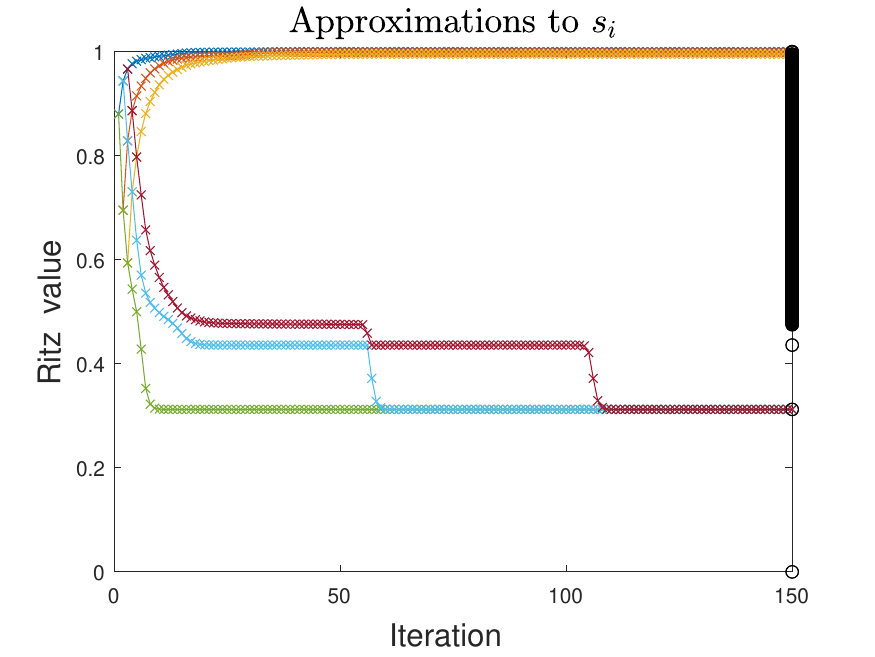}}\hspace{-0.0mm}
	\subfloat
	{\label{fig:2b}\includegraphics[width=0.48\textwidth]{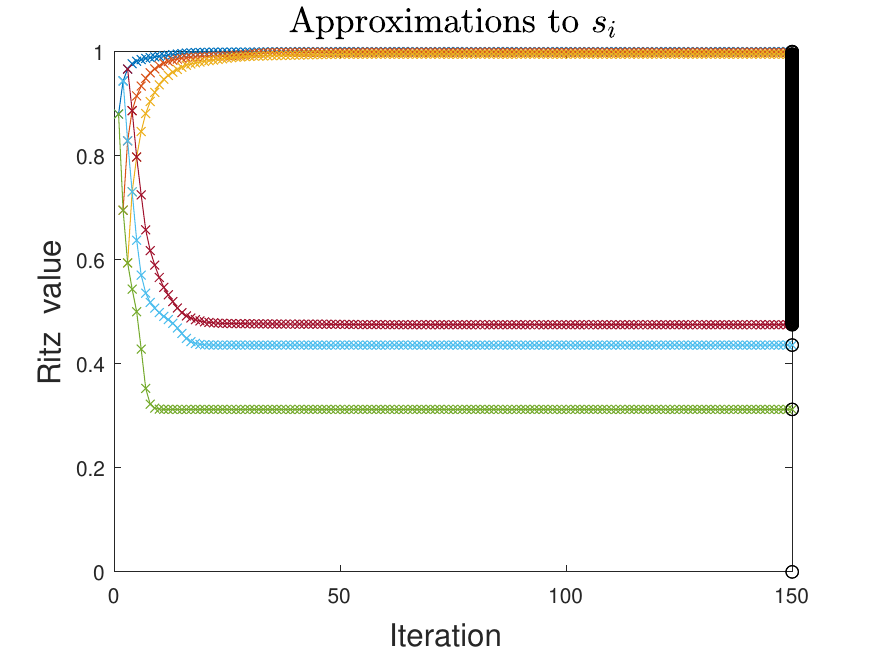}}
	\vspace{-2mm}
	\subfloat 
	{\label{fig:2c}\includegraphics[width=0.48\textwidth]{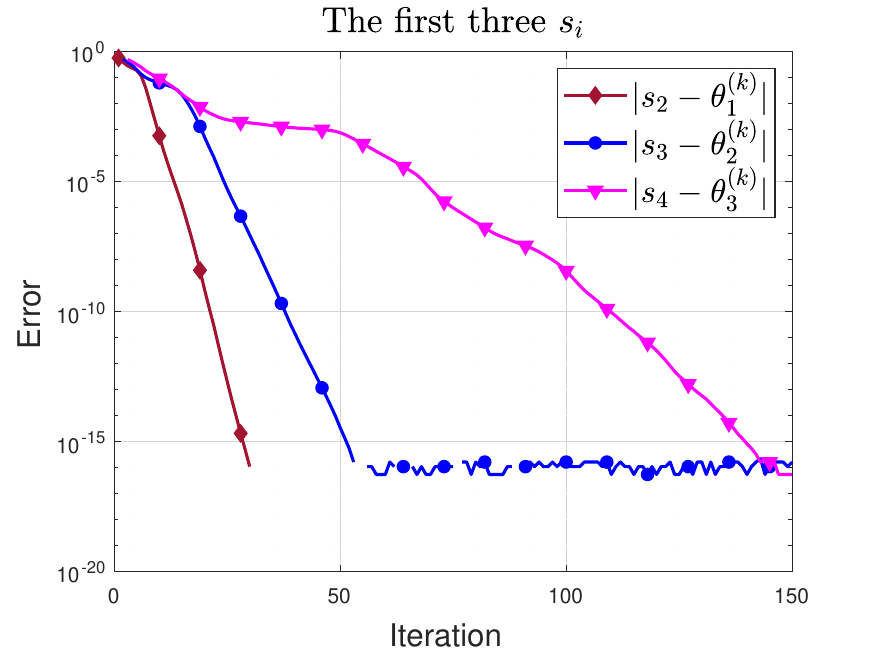}}\hspace{-0.0mm}
	\subfloat
	{\label{fig:2d}\includegraphics[width=0.48\textwidth]{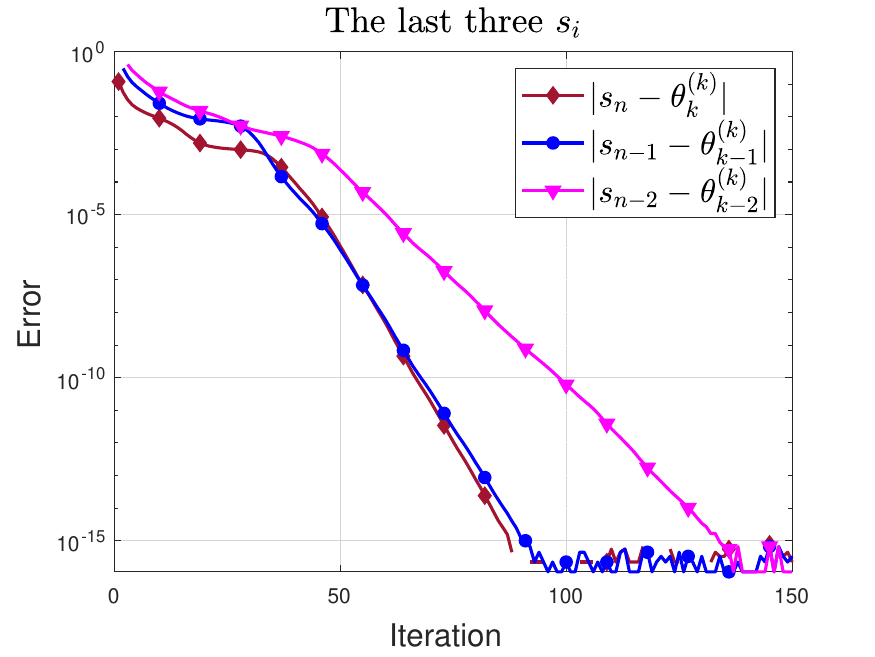}}
	\vspace{-2mm}
	\caption{Convergence and accuracy of approximations to $s_i$. Top: convergence of Ritz values $\theta_{i}^{(k)}$ to largest/smallest $s_i$ by \textsf{gGKB\_GSVD} without reorthogonalization (left) and with full reorthogonalization (right). Bottom: error curves (full reorthogonalization).}
	\label{fig2}
\end{figure}

We also test using the \textsf{gGKB} of $\calL$ to approximate several largest and smallest $s_i$. The convergence behavior of the Ritz values and error curves are plotted in \Cref{fig2}. In addition to the findings closely resembling those depicted in \Cref{fig1}, there are two additional insights. First, we find that the first three smallest Ritz values converge to $s_2,s_3,s_4$ instead of $s_1,s_2,s_3$. This is because to $s_1=0$, which can not be converged upon by Ritz values, as revealed by \Cref{thm:gsvd_exact} and \Cref{thm:conv}. Therefore, we should use \textsf{gGKB} of $\calA$ to compute the generalized singular values with value $\infty$. Second, we find from the bottom two subfigures that those smallest $s_i$ can be approximated more quickly than the largest ones, due to their well-separated locations. Notice that those smallest $s_i$ correspond to the largest generalized singular values. Therefore, for computing GSVD components corresponding to those smallest generalized singular values, using \textsf{gGKB\_GSVD} based on the \textsf{gGKB} of $\calL$ may be a better choice.

\paragraph*{Example 2}
The matrix $A$ named {\sf well1850} is taken from the SuiteSparse matrix collection \cite{Davis2011}, and the matrix $L$ is set as 
\begin{equation*}
	L=\begin{pmatrix}
		1.1 & -1 &  &  &  \\
		& 1.1 & -1 &  &  \\
		&  & \ddots & \ddots &  \\
		&  &  & 1.1  & -1 \\
	\end{pmatrix} \in \mathbb{R}^{(n-1)\times n}.
\end{equation*}
This is a regular matrix pair. We use the Matlab build-in function \texttt{gsvd.m} to compute the full GSVD of $\{A, L\}$ as the baseline of comparison. 

\begin{figure}[htbp]
	\centering
	\subfloat
	{\label{fig:3a}\includegraphics[width=0.48\textwidth]{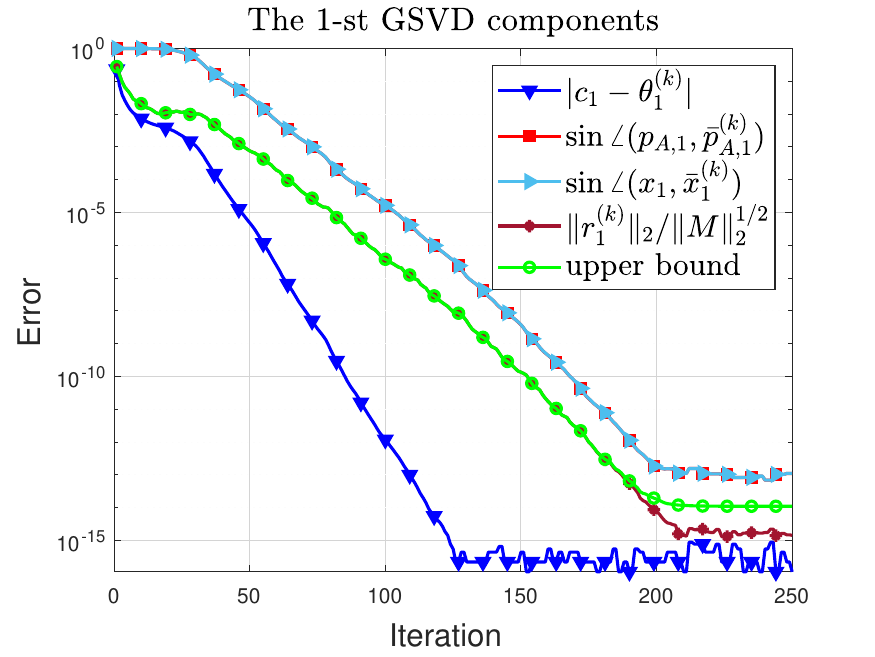}}\hspace{-0.0mm}
	\subfloat
	{\label{fig:3b}\includegraphics[width=0.49\textwidth]{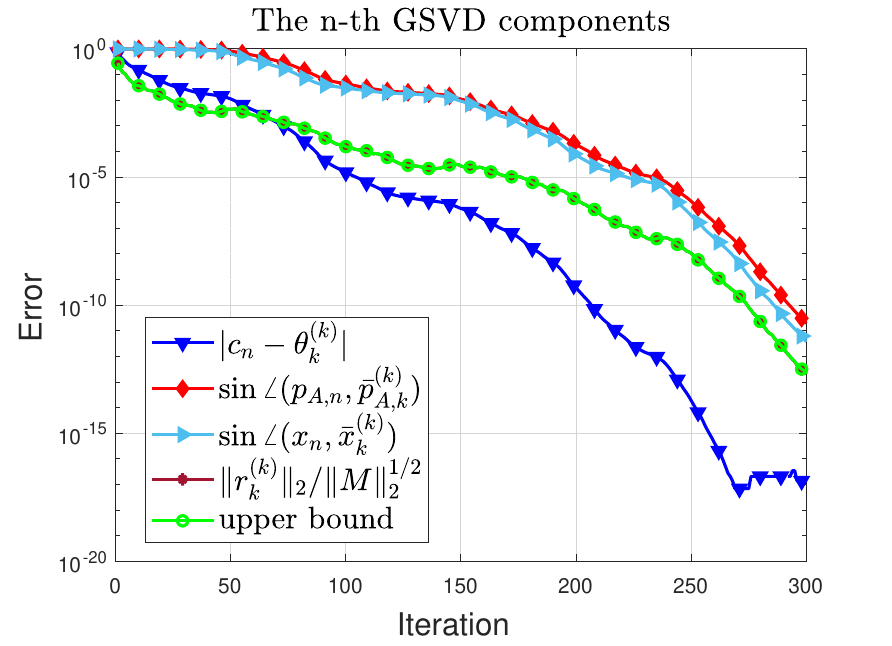}}
	\vspace{-2mm}
	\caption{Error curves of the approximate GSVD components by \textsf{gGKB\_GSVD} and relative residual norm with its upper bound. Left: approximations to the 1-st GSVD components. Right: approximations to the n-th GSVD components.}
	\label{fig3}
\end{figure}

In this experiment, we test the performance of \textsf{gGKB\_GSVD} for computing the first and $n$-th GSVD components of $A$ for a matrix pair with nonsquare matrices. We only show the results for the \textsf{gGKB} of $\calA$ and omit the similar results for the \textsf{gGKB} of $\calL$. Full reorthogonalization is used and $M^{-1}$ is computed directly. The errors for the approximated generalized singular vectors are measured by $\sin\angle(x, y)$ between two vectors. We also plot the variation of the relative residual norm and its upper bound $\alpha_{k+1}\beta_{k+1}|e_{k}^{\top}h_{i}^{(k)}|$. \Cref{fig3} shows that \textsf{gGKB\_GSVD} can approximate very well the two group extreme GSVD components, with final accuracy around $\mathcal{O}(\mathbf{u})$. The upper bound $\alpha_{k+1}\beta_{k+1}|e_{k}^{\top}h_{i}^{(k)}|$ exhibits a nearly identical decreasing trend as the relative residual norm. Therefore, it is a highly suitable quantity to be employed in the stopping criterion. We also observe that the convergence to the first GSVD components is faster than the convergence to the and $n$-th.

\paragraph*{Example 3}
The matrix pair $\{A, L\}$ is constructed as follows. Set $m=n=p=1000$ and set $r=900$. Let $C_{A}=(\Sigma_{A}, \mathbf{0})$ with $C_A=\mathrm{diag}(\{c_i\}_{i=1}^{r})$, where $c(1)=0.99, c(2)=0.98$, $c(3:r-4)=\texttt{linspace(0.96,0.06,r-4)}$ and $c(r-1)=0.04, c(r-1)=0.02$, and let $S_{L}=(\Sigma_{L}, \mathbf{0})$ with $\Sigma_L=\mathrm{diag}(\{s_i\}_{i=1}^{r})$ and $s_i=(1-c_{i}^2)^{1/2}$. Then let  $W=\texttt{gallery(`orthog',n,2)}$ be an orthogonal matrix and $D=\texttt{diag(linspace(1,10,n))}$. Finally let $A=C_AW^{\top}D$ and $L=S_LW^{\top}D$. By the construction, we have $\mathrm{rank}((A^{\top}, L^{\top})^\top) = r<n$, and the nontrivial GSVD components are $c_i, \ s_i$ and the $i$-th columns of $I_n$, $I_n$ and $D^{-1}W$ for $1\leq i \leq r$. For each nontrivial $x_i$, we compute $\calP_{\calR(M)}x_i=MM^{\dag}x_i$ to get the corresponding right generalized singular vector belonging to $\calR(M)$. We use \textsf{gGKB\_GSVD} to compute $x_i$ that belongs to $\calR(M)$. 

\begin{figure}[!htbp]
	\centering
	\subfloat
	{\label{fig:4a}\includegraphics[width=0.48\textwidth]{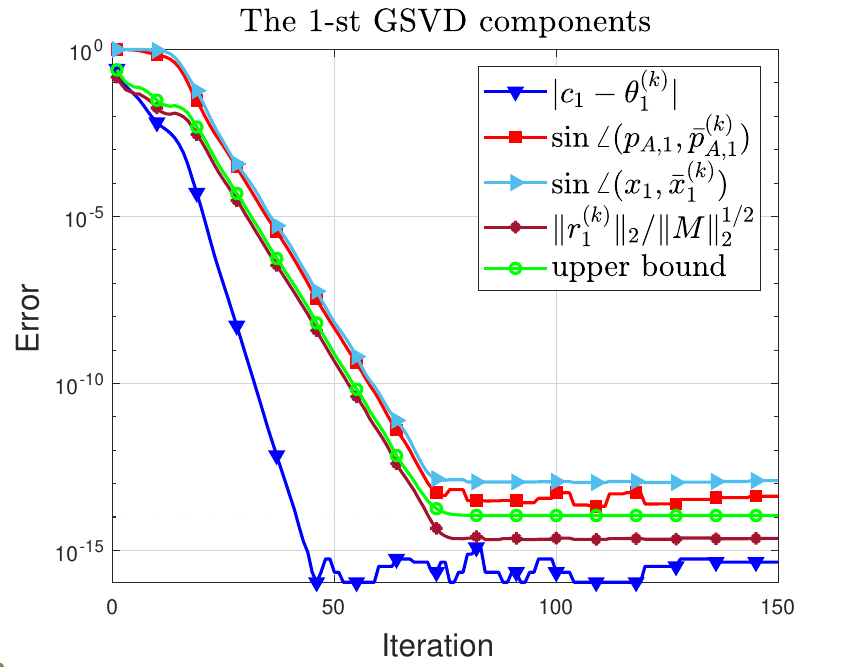}}\hspace{-0.0mm}
	\subfloat
	{\label{fig:4b}\includegraphics[width=0.50\textwidth]{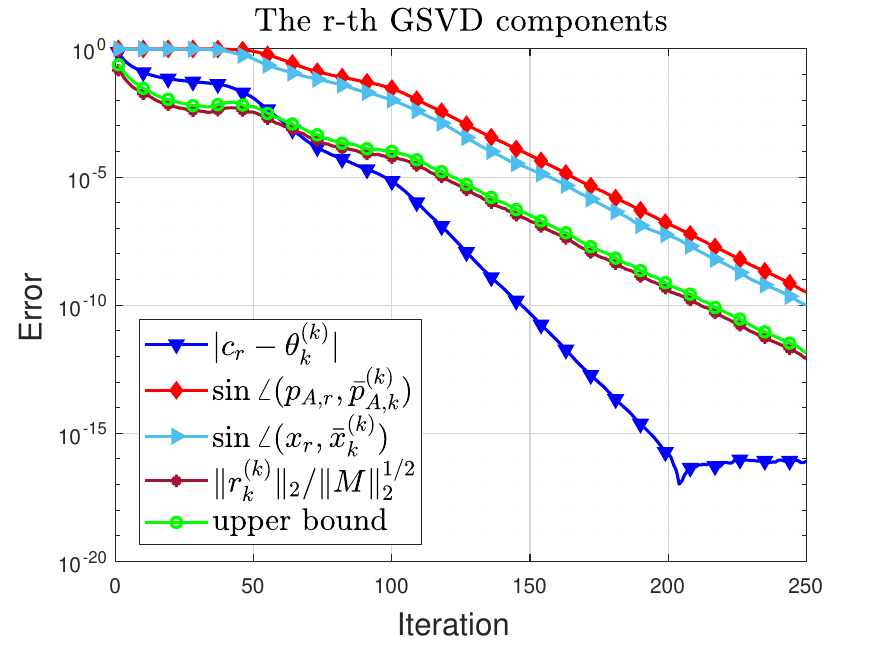}}
	\vspace{-2mm}
	\caption{Error curves of the approximate GSVD components by \textsf{gGKB\_GSVD} and relative residual norm with its upper bound, where $\mathrm{rank}((A^{\top}, L^{\top})^\top) = r<n$. Left: approximations to the 1-st GSVD components. Right: approximations to the r-th GSVD components.}
	\label{fig4}
\end{figure}

In this experiment, we test the performance of \textsf{gGKB\_GSVD} for computing the first and $r$-th GSVD components for a nonregular matrix pair. We directly compute $M^{\dag}$ and use full reorthogonalization for \textsf{gGKB}. \Cref{fig4} shows very good performance of the algorithm: (1) the two extreme GSVD components $(c_i, p_{A,i}, x_{i})$ for $i=1,r$ can be approximated with final accuracy around $\mathcal{O}(\mathbf{u})$; (2) the relative residual norm and its upper bound $\alpha_{k+1}\beta_{k+1}|e_{k}^{\top}h_{i}^{(k)}|$ follow nearly identical decreasing curves, with both eventually stabilizing at a level around $\mathcal{O}(\mathbf{u})$. Again, we observe that the convergence to the first GSVD components is faster than the convergence to the $r$-th.

\begin{figure}[!htbp]
	\centering
	\subfloat
	{\label{fig:5a}\includegraphics[width=0.48\textwidth]{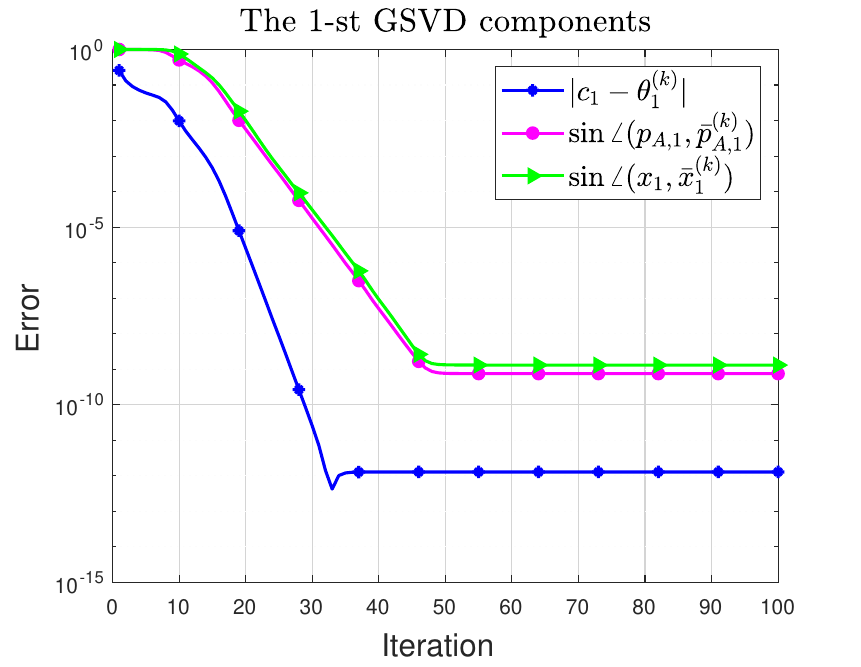}}\hspace{-0.0mm}
	\subfloat
	{\label{fig:5b}\includegraphics[width=0.50\textwidth]{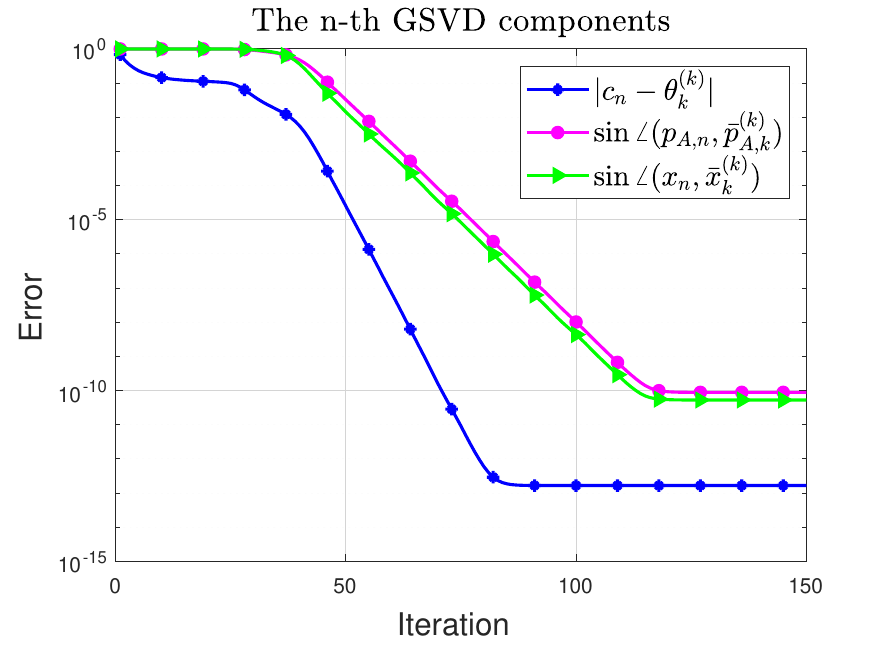}}
	\vspace{-2mm}
	\subfloat 
	{\label{fig:5c}\includegraphics[width=0.48\textwidth]{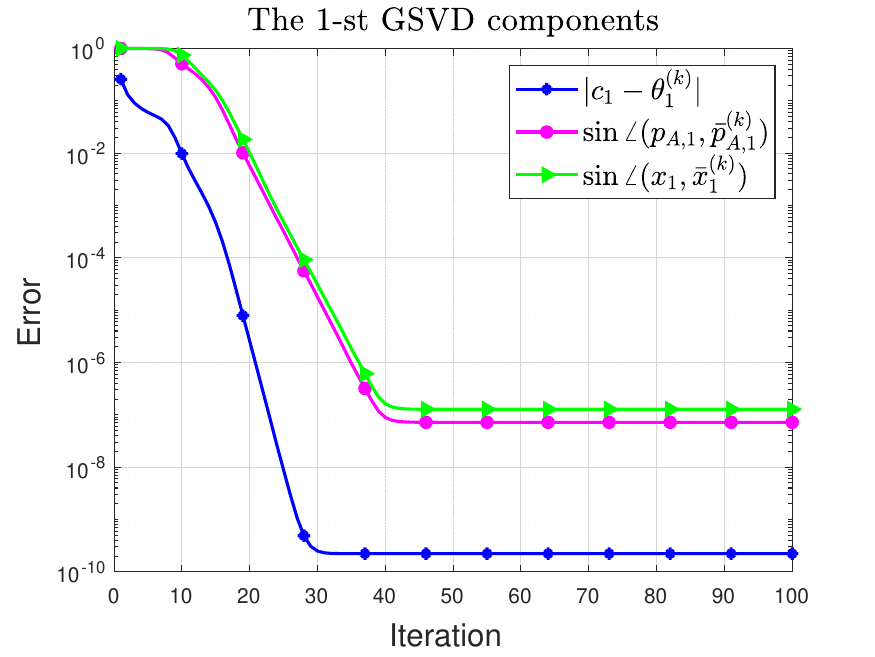}}\hspace{-0.0mm}
	\subfloat
	{\label{fig:5d}\includegraphics[width=0.50\textwidth]{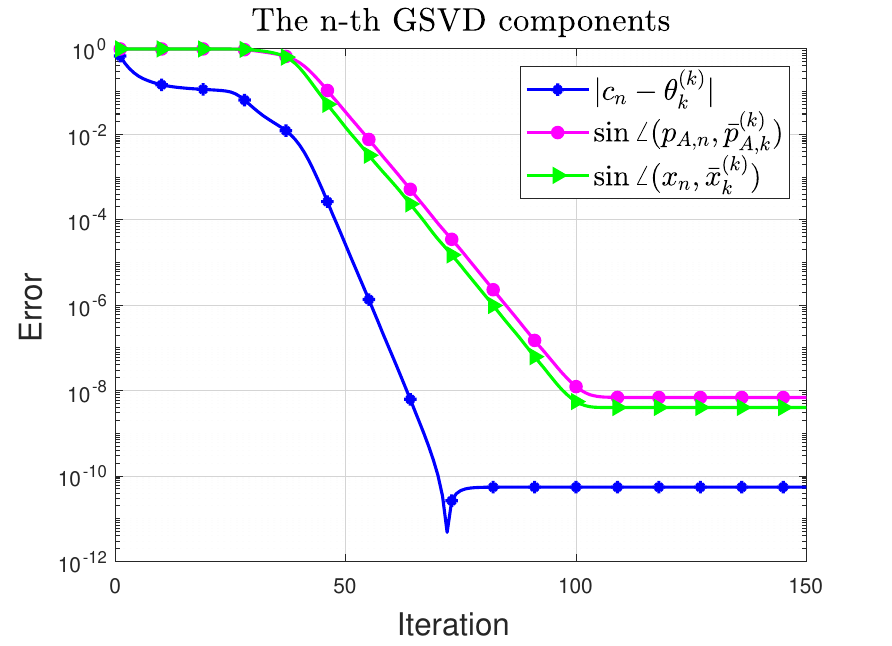}}
	\vspace{-2mm}
	\caption{Accuracy of computed GSVD components by \textsf{gGKB\_GSVD}, where $s=M^{\dag}\bar{s}$ at each \textsf{gGKB} iteration is computed by solving $\min_{s}\|Ms-\bar{s}\|_2$ using \texttt{lsqr.m} with different stopping tolerance \texttt{tol}. Top: \texttt{tol}=$10^{-10}$. Top: \texttt{tol}=$10^{-8}$.}
	\label{fig5}
\end{figure}

\paragraph*{Example 4}
The matrix pair $\{A, L\}$ is constructed as follows. Set $m=n=p=5000$. Let $C_A=\mathrm{diag}(\{c_i\}_{i=1}^{r})$ with $c(1)=0.99, c(2)=0.97$, $c(3:n-4)=\texttt{linspace(0.95,0.15,n-4)}$ and $c(n-1)=0.1, c(n)=0.05$. Let $S_{L}=\mathrm{diag}(\{s_i\}_{i=1}^{n})$ with $s_i=(1-c_{i}^2)^{1/2}$. Then let $W=\texttt{gallery(`orthog',n,2)}$ be an orthogonal matrix and $D=\texttt{diag(linspace(1,10,n))}$. Finally let $A=C_AW^{\top}D$ and $L=S_LW^{\top}D$. By the construction, $\{A,L\}$ is a regular matrix pair, and the $i$-th GSVD components are $c_i, \ s_i$ and $i$-th columns of $I_n$, $I_n$ and $D^{-1}W$.

We use this experiment to demonstrate the impact of inaccuracy in the computation of $M^{\dag}\bar{s}$ on the final accurate of the approximate GSVD components. We use the Matlab build-in function \texttt{lsqr.m} to solve \eqref{ls_gGKB} iteratively with stopping tolerance $\texttt{tol}=10^{-10}, 10^{-8}$ at each iteration of \textsf{gGKB}, respectively. \Cref{fig5} shows the decrease of relative errors of the first and $n$-th approximate GSVD components with the two stopping tolerances. We observe that the computational accuracy of $M^{\dag}\bar{s}$ significantly affects the final accuracy of both the generalized singular values and vectors. As the computational accuracy deteriorates, so does the final accuracy of the computed GSVD components. Further theoretical investigation into this issue should be conducted in future research.

\section{Conclusion and outlook}\label{sec7}
Based on the theory of singular value expansion (SVE) of linear compact operators, we have provided a new understanding of the GSVD of $\{A, L\}$ with $A\in\mathbb{R}^{m\times n}$ and $L\in\mathbb{R}^{p\times n}$. By defining the positive semidefinite matrix $M=A^{\top}A+L^{\top}L$, we have shown that: (1) the trivial GSVD components $\{x_i\}$ form a basis for $\calN(M)$ and any nontrivial $x_i$ belongs to the coset $\bar{x}_i+\calN(M)$, where $\bar{x_i}\in\calR(M)$ is a nontrivial GSVD component; (2) the nontrivial GSVD components of $A$ and $L$ are just the SVEs of the linear operators $\calA: (\calR(M),\langle\cdot,\cdot\rangle_{M})\rightarrow(\mathbb{R}^{m},\langle\cdot,\cdot\rangle_{2}), \ v \mapsto Av$ and $\calL: (\calR(M),\langle\cdot,\cdot\rangle_{M})\rightarrow(\mathbb{R}^{p},\langle\cdot,\cdot\rangle_{2}), \ v \mapsto Lv$, respectively. As a direct application of this result, we have developed an operator-type Golub-Kahan bidiagonalization (GKB) for $\calA$ and $\calL$, leading to a novel generalized GKB (\textsf{gGKB}) process. We have used the GSVD of $\{A, L\}$ to study basic properties of \textsf{gGKB} and proposed the \textsf{gGKB\_GSVD} algorithm to compute several nontrivial extreme GSVD components of large-scale matrix pairs. Preliminary results about convergence and accuracy of \textsf{gGKB\_GSVD} for GSVD computation have been provided, and numerical experiments are presented to demonstrate the effectiveness of this method.

The idea of this paper offers potential directions for developing new algorithms for large-scale GSVD computation. Note that the SVE of $\calA$ or $\calL$ can be treated as a ``weighted'' SVD, where the weight matrix $M$ induces a non-Euclidean inner product. Therefore, existing SVD algorithms based on Krylov subspace projection may be modified to approximate the SVE and consequently, the nontrivial GSVD components.


\bibliographystyle{abbrv}
\bibliography{references}

\end{document}